\documentclass{article}

\usepackage{arxiv}
\usepackage[round]{natbib}

\usepackage{amsthm,amsmath,amssymb}
\usepackage{color-edits}
\usepackage{thm-restate}
\usepackage{hyperref}
\usepackage{cleveref}
\usepackage{fullpage}
\usepackage{enumerate}
\usepackage{xcolor}
\usepackage{mathtools}

\usepackage[ruled,vlined]{algorithm2e}

\DeclareMathOperator*{\argmin}{arg\,min}

\renewcommand{\c}[1]{\ensuremath{\mathcal{#1}}} %
\newcommand{\Ex}[1]{\ensuremath{\mathbb{E}\left[#1\right]}}  %

\newcommand{\zpn}{zero-preserving noisy~}

\newcommand{\ac}{anti-concenterated~}
\newcommand{\acs}{AC}

\newcommand{\calD}{{\mathcal{D}}}

\newcommand{\calO}{{\mathcal{O}}}
\newcommand{\calX}{{\mathcal{X}}}

\newcommand{\calC}{{\mathcal{C}}}
\newcommand{\calF}{{\mathcal{F}}}
\newcommand{\calR}{{\mathcal{R}}}

\newcommand{\calN}{{\mathcal{N}}}
\newcommand{\calE}{{\mathcal{E}}}

\newcommand{\calL}{\mathcal{L}}
\newcommand{\calA}{\mathcal{A}}

\newcommand{\Ball}{\mathbb{B}}
\newcommand{\Real}{{\mathbb{R}}}

\newcommand{\twonorm}[1]{\left\| #1\right\|_2}

\newcommand{\paren}[1]{\left( #1\right)}
\newcommand{\abs}[1]{\left| #1\right|}
\newcommand{\inner}[2]{\left\langle #1, #2\right\rangle}
\newcommand{\gradest}{{\widehat{\nabla}}}
\newcommand{\chisq}{{\chi^2}}

\newcommand{\define}{\ensuremath{:=}}

\newtheorem{theorem}{Theorem}[section]

\newtheorem{example}[theorem]{Example}
\newtheorem{lemma}[theorem]{Lemma}

\newtheorem{definition}[theorem]{Definition}

\renewcommand{\c}[1]{\ensuremath{\mathcal{#1}}} %
\newcommand{\bb}[1]{\ensuremath{\mathbb{#1}}} %

\title{On the Inherent Privacy of Zeroth Order Projected Gradient Descent}

\date{} 					%

\author{{\hspace{1mm}Devansh Gupta} \\
	Department of Computer Science\\
	University of Southern California\\
	Los Angeles, CA 90007\\
	\texttt{guptadev@usc.edu} \\
        \And
	{\hspace{1mm}Meisam Razaviyayn} \\
	Departments of ISE, CS, ECE, and QCB\\
	University of Southern California\\
	Los Angeles, CA 90007\\
	\texttt{razaviya@usc.edu} \\
    \And
	{\hspace{1mm}Vatsal Sharan} \\
	Department of Computer Science\\
	University of Southern California\\
	Los Angeles, CA 90007\\
	\texttt{vsharan@usc.edu} \\
}

\renewcommand{\headeright}{}
\renewcommand{\undertitle}{}
\renewcommand{\shorttitle}{On the Inherent Privacy of Zeroth Order Projected Gradient Descent}

\begin{document}
\maketitle

\begin{abstract}
  Differentially private zeroth-order optimization methods have recently gained popularity in private fine tuning of machine learning models due to their reduced memory requirements. Current approaches for privatizing zeroth-order methods rely on adding Gaussian noise to the estimated zeroth-order gradients. However, since the search direction in the zeroth-order methods is inherently random, researchers including \citet{tang2024privatefinetuninglargelanguage} and \citet{zhang2024dpzero} have raised an important question: is the inherent noise in zeroth-order estimators sufficient to ensure the overall differential privacy of the algorithm? This work settles this question  for a class of oracle-based optimization algorithms where the oracle returns zeroth-order gradient estimates. In particular, we show that for a fixed initialization, there exist strongly convex objective functions such that running (Projected) Zeroth-Order Gradient Descent (ZO-GD) is not differentially private. 
Furthermore, we show that even with random initialization and without revealing (initial and) intermediate iterates, the privacy loss in ZO-GD can grow superlinearly with the number of iterations when minimizing convex objective functions.
  
\end{abstract}

\section{Introduction}
\label{sec:intro}

The fine-tuning of pretrained large language models (LLMs) has demonstrated state-of-the-art performance across a range of downstream applications. However, two main challenges hinder the wide adoption of these models: 
the substantial memory requirements of gradient-based optimizers used for fine-tuning and the critical need to protect the privacy of domain-specific fine-tuning data. As fine-tuning LLMs grows increasingly memory-intensive, a range of strategies has emerged to address this issue. In particular, zeroth-order (ZO) optimization methods recently have gained traction due to their memory efficiency, as they do not require explicit gradient computations. Instead, the zeroth-order gradients can be computed using  forward step only, significantly reducing memory use compared to gradient computation.

In a pioneering approach, \citet{malladi2023mezo} introduced a memory-efficient technique for fine-tuning LLMs using two-point Simultaneous Perturbation Stochastic Approximation (SPSA) estimators \citep{spsa}, enabling large model fine-tuning on memory-limited devices.
Since then, zeroth-order methods have gained  popularity in dealing with large machine learning models due to their memory efficiency and favorable upper bounds on gap from optimality under certain conditions on the Hessian of the objective function \citep{zhang2024revisitingzerothorderoptimizationmemoryefficient, zhang2024dpzero, guo2024zerothorderfinetuningllmsextreme}.

Another major concern in training LLMs is \textit{privacy}. As large parameterized models are increasingly used in sensitive data applications, these models must protect sensitive information, especially given privacy regulations like the E.U. General Data Protection Regulation and the California Consumer Privacy Act. This requirement led to significant research into differential privacy (DP), a robust framework ensuring that machine learning models do not compromise the privacy of their contributors \citep{DworkRothblumVadhan2010}. As a result, there has been a growing focus on developing methods that fine-tune LLMs while adhering to differential privacy standards, leading to numerous theoretical advancements in private optimization \citep{boosting, pmlr-v19-chaudhuri11a, thresholdfunctions, NEURIPS2019_700fdb2b, block2024oracleefficientdifferentiallyprivatelearning, JMLR:v12:chaudhuri11a, BassilySmith2014, pmlr-v235-lowy24b, aroraprivacy, optlinear, bassilyoptimal, pmlr-v178-gopi22a, altschuler2022privacy, resque, generalnorms} and practical applications in the industry \citep{abadipaper, opacus, 10.1145/2660267.2660348, NIPS2017_253614bb, Rogers_Subramaniam_Peng_Durfee_Lee_Kancha_Sahay_Ahammad_2021}. Nevertheless, most existing work in this area has focused on first-order optimization/training algorithms, highlighting the need to explore differentially private zeroth-order optimization techniques to combine memory efficiency with privacy protection.

Motivated by the memory efficiency and empirical success of ZO methods in fine-tuning LLMs, \citet{tang2024privatefinetuninglargelanguage} and \citet{zhang2024dpzero} introduced differentially private and memory-efficient algorithms based on ZO optimization techniques. Both noted that the inherent noise in zeroth-order gradient estimates might contribute to privacy protection. As a result, they highlighted that the inherent noise in the ZO estimators was not considered in their privacy analyses and posed the following open question:

\vspace{0.2cm}

\fbox{
    \begin{minipage}{\columnwidth - 14pt}
    \textit{\textbf{Open Problem:} Is {additive (Gaussian) noise necessary for ensuring privacy for ZO Stochastic Gradient Descent (SGD)?}  }
    \end{minipage}
}

\vspace{0.2cm}

In this work, we address this question through the following key \textbf{contributions}:

\begin{enumerate}
    \item We propose a class of oracles that generalizes multiple point zeroth-order estimators and show that any estimator in our class is not differentially private.
    \item We show that for a generalized setting which subsumes the algorithms proposed in \citet{tang2024privatefinetuninglargelanguage} and \citet{zhang2024dpzero}, the ZO method is not private on its own and the presence of additive noise is necessary to preserve privacy of the algorithm. This answers the open problem posed by \cite{tang2024privatefinetuninglargelanguage} and \citet{zhang2024dpzero}. 
    \item We further show that, even with random initialization and without disclosing intermediate iterates, optimizing specific types of objectives using ZO-GD results in a superlinear increase in privacy loss as the number of iterations grows.  This finding suggests that, despite random initialization and privacy amplification through iterations, the inherent randomness of zeroth-order methods is insufficient to guarantee meaningful privacy in practice.
\end{enumerate}

\subsection{Related Work and Existing Results}
\paragraph{ZO Optimization.} The idea of minimizing functions based on function evaluations has origins from control theory \citep{spsa}. Such an approach is useful when obtaining gradients is either impossible or too costly, making ZO methods favorable for minimizing non-smooth or even discontinuous functions. For example, \citet{NIPS2012_e555ebe0, Nesterov2015RandomGM} gave upper bounds on the gap between the optimal solutions and returned solutions after a finite number of iterations in the case of convex functions. They relied on gradient oracles based on the finite difference method and show that these oracles estimate the gradient of the smoothed version of the function, as we discuss in Section \ref{sec:prelim}. To understand the fundamental limit on the performance of such algorithms, \citet{pmlr-v28-shamir13} showed the existance of convex functions with a lower bound on the optimality gap for any ZO algorithm that queries a single point per iteration. \citet{duchioptzero} extended this result to algorithms that query multiple points per iteration. Recently \citet{malladi2023mezo, pmlr-v195-yue23b} provided nearly-dimension-independent upper bounds on the optimality gap for ZO algorithms, primarily depending on a measure termed as the effective dimension of the problem which relates to some notion of a local rank of the hessian of the function. However, in the worst case, this effective dimension is equal to the actual dimension of the problem, restoring the lower bounds obtained for their specific cases in \citet{pmlr-v28-shamir13} and \citet{duchioptzero}. However, from an application perspective, the primary reason ZO methods have gained traction is that many over-parameterized models, such as pretrained Large Language Models (LLMs), are shown to have a low effective dimension \citep{hu2022lora}. \citet{malladi2023mezo, zhang2024dpzero} leverage this insight to justify the effectiveness of ZO algorithms in fine-tuning LLMs.

\paragraph{Private Convex Optimization.} On the other hand, optimization is the foundation of modern large-scale machine learning, making it essential to understand private optimization to fully grasp privacy in machine learning. Along this step, \citet{BassilySmith2014} gave the first tight upper and lower bounds for minimizing convex and strongly convex functions. Their idea for minimizing smooth functions was making Stochastic Gradient Descent (SGD) differentially private by updating their parameters with a noisy version of the gradient estimate. This algorithm achieves optimal excess risk gap (upto logarithmic terms) in private empirical risk minimization (ERM) \citep{BassilySmith2014} and stochastic convex optimization (SCO) for smooth objectives \citep{bassilyoptimal}. Moreover, a small (yet effective) modification to this algorithm  achieves optimal rates in just one pass over the entire data \citep{optlinear}. There have been further works which have also solved the problem of differentially private convex optimization under non-smooth conditions \citep{bassilyoptimal, NEURIPS2021_211c1e0b, optlinear} and general norms \citep{pmlr-v139-asi21b, generalnorms}. Other than private SGD, \citet{BassilySmith2014} proposed an exponential-mechanism-based algorithm which achieved optimal risk bounds up to constant factors for $\varepsilon$-DP minimization. \citet{pmlr-v178-gopi22a} further extended the exponential mechanism to $(\varepsilon, \delta)$ DP SCO and obtained upper bounds on the excess population loss. Recently, \citet{resque} proposed a new estimator specifically made for accelerated private optimization giving improvements on the query complexity of private SCO.

\textbf{Private ZO optimization} is an area which is relatively new and there have been recent works which give guarantees on private ZO optimization for smooth convex problems~\citep{malladi2023mezo}, smooth non-convex problems~\citep{zhang2024dpzero, tang2024privatefinetuninglargelanguage} and non-smooth non-convex problems \citep{zhang2024private}. 

On the other hand, the work on inherent privacy in zeroth-order optimization is relatively sparse. \citet{tang2024privatefinetuninglargelanguage} provided an \textit{empirical} evaluation of the privacy of Projected ZO-SGD using a privacy accounting technique based on membership-inference attacks on ML models \citep{shokri2017membership}. They empirically showed that ZO SGD without any additive noise gave almost the same privacy as regular SGD with additive noise corresponding to $\varepsilon = 10$ and $\delta = 10^{-5}$, giving some \textit{experimental evidence} to conjecture the possible presence of the inherent privacy in ZO optimization. However, we prove that there exist worst case objectives and datapoints where popular ZO optimization algorithms (including the algorithm studied in \citet{tang2024privatefinetuninglargelanguage}) are not private.

\section{Problem Setting and Preliminaries}
\label{sec:prelim}
\paragraph{Notation.} We use $\twonorm{\cdot}$ for the Euclidean $L_2$ norm. We denote $\bb{P}[E]$ as the probability of any event $E$. $\bb{E}[{X}]$ denotes the expectation of any random variable ${X}$. For any distribution $\calD$, $supp(\calD)$ represents the support of the distribution, $\calN(0, \sigma^2I_d)$ is the isotropic normal distribution with mean 0 and covariance $\sigma^2 I_d$. For any set S, $Unif(S)$ is the uniform distribution over the set $S$. $D\bb{B}^d = \{x \in \Real^d: \twonorm{x} \leq D\}$. $\mathbf{0}_d$, $\mathbf{1}_d$, and $\mathbf{e}_i$ represent $d$ dimensional all zero, all one, and the $i^{th}$ standard basis vectors, respectively. For a vector $v \in \Real^d$, $\{v\}_i$ represents the $i^{th}$ coordinate of the vector; $[n]$ denotes the set of natural numbers less than or equal to $n$. $\Pi_D(x) = \argmin_{u \in D}\twonorm{x-u}$ is used to denote the projection of $x \in \Real^d$ onto the set $D$. For any function $f$, $dom(f)$ represents the domain of the function $f$.

\paragraph{Problem Setting.} We consider the problem of minimizing the empirical loss function with respect to the dataset $\calD = \{d_1, d_2, ..., d_n\}$ given a closed convex set $\calC \subseteq \Real^d$
\begin{align*}
    \min_{w \in \calC}~\calL(w; \calD)
\end{align*}
where $\calL(w;\cdot): \calC \rightarrow \Real$ is convex and $L$-Lipschitz. There can be additional assumptions on $\calL(w;\cdot)$ such as $\Delta$ strong convexity. Recall that a function $g$ is called $L$-Lipschitz if $\twonorm{g(x) - g(y)} \leq L\twonorm{x-y}$ for all $x,y \in \calC$ and it is called $\mu$-Strongly Convex when for all $x, y \in \calC$, $f(x) \geq f(y) + \inner{z}{y - x} + \frac{\mu}{2}\twonorm{y-x}^2$ where $z$ is any subgradient of $f$. Next, we define the notion of Differential Privacy.

\begin{definition}[Differential Privacy \citep{originalDP}]
    \label{def:dp}
    Two datasets of the same size are neighbouring if $|\calD \Delta \calD'| = 2$ where $\Delta$ represents the symmetric set difference. Let $\varepsilon \geq 0, ~\delta \in [0, 1).$ A randomized algorithm $\calA$ is $(\varepsilon, \delta)$-differentially private (DP) if for all pairs of neighbouring data sets $\calD, \calD'$, we have
    \begin{equation}
    \label{eq: DP}
    \bb{P}(\calA(\calD) \in O) \leq e^\varepsilon \bb{P}(\calA(\calD') \in O) + \delta, 
    \end{equation}
for any measurable set~$O$. 
\end{definition}

\paragraph{ZO Stochastic Gradient Descent.} The algorithmic framework that we will be operating under is the \textit{projected stochastic ZO descent} as given in Algorithm \ref{alg:proj_zo_sgd}. Algorithm \ref{alg:proj_zo_sgd} requires access to \textit{zeroth order oracles} to query the update direction at each step. We define such oracles below. 
\begin{definition}
    \label{def:zeroorderoracle}
    A zeroth order oracle $\calO$ is an oracle which takes a function $f: \Real^d \rightarrow \Real$, a single point $w \in \Real^d$ and returns a probability measure over a subset of $\Real^d$ using only function evaluations on different points depending on $w$.
\end{definition}

Notably, ZO oracle may require multiple function evaluations, which may not necessarily depend on the point $w$. In the design of ZO algorithms, these oracles are typically implemented using well-established ZO estimators. Below, we provide examples of such estimators.

\begin{example}
    \label{ex:elemestimators}
    Here we list three popular ZO estimators: Simultaneous Perturbation Stochastic Approximation (SPSA) introduced by \citet{spsa}, Finite Difference (FD) introduced by \citet{Nesterov2015RandomGM}, and Single Point (SP) estimator introduced by \citet{singlepoint}.
    \begin{enumerate}
        \item $SPSA_{\xi}(f, w)$: Sample ${Z} \sim \calN(0, I_d)$ and return 
        $\widehat{\nabla}_1 f_\xi(w) \define \frac{f(w + \xi {Z}) - f(w - \xi {Z})}{2\xi}{Z}$.
        \item $FD_{\xi}(f, w)$: Sample ${Z} \sim \calN(0, I_d)$ and return $\widehat{\nabla}_2 f_\xi(w) \define \frac{f(w + \xi {Z}) - f(w)}{\xi}{Z}$.
        \item $SP_\xi(f, w)$: Sample ${Z} \sim \calN(0, I_d)$ and return $\widehat{\nabla}_3 f_\xi(w) \define \frac{d}{\xi}f\left(w + \xi {Z}\right){Z}$
    \end{enumerate}
\end{example}

 The above mentioned estimators are widely used in various zeroth order optimization and control algorithms. They are unbiased estimators of the gradient of a smoothed version of the function $f$, defined as $f_\xi(x) = \bb{E}_{{Z} \sim \calN(0, I_d)}\left[f(x + \xi{Z})\right]$. In other words, $\bb{E}\left[\widehat{\nabla} f_\xi(w)\right] = \nabla f_\xi(x)$ \citep{Nesterov2015RandomGM, nemrosyud}. 

Since the aforementioned estimators are randomized, several works, such as \citet{zhang2024private}, \citet{malladi2023mezo}, and \citet{duchioptzero}, have focused on strategies to reduce the variance of the updates obtained through these estimators, by taking the mean of different samples. Building on these randomized estimators, multi-point estimators can be defined aggregate information across multiple points. We provide an example of such aggregation below.

\begin{example}
    \label{ex:mean}
    For any randomized estimator $\calE$, the mean extension of the estimator $M^\calE_m(f, w)$ works as follows: Sample i.i.d. $U_1, U_2, ..., U_m \sim \calE(f, w)$ and return $\frac{1}{m} \sum_{i=1}^m U_i$.
\end{example}

\begin{algorithm}[t]
\caption{Projected Stochastic ZO Descent}
\label{alg:proj_zo_sgd}
    Given number of steps $T$, initialization distribution $\calR_{init}$, ZO oracle $\calO$, and constraint set $\calD$ \\
    Sample $w_0 \sim \calR_{init}$ \\
    \For{$t \gets 1, ..., T$}{
        Draw $\gradest \calL(w_{t-1}; \calX)$ from the distribution $\calO(\calL(\cdot; \calX), w_{t-1})$\\
        $w_{t} \gets \Pi_\calD\paren{w_{t-1} - \eta \gradest \calL(w_{t-1}; \calX)}$\\
    }
\end{algorithm}
 
\begin{example}
    Using the construction mentioned in Example~\ref{ex:mean}, one can further define $M_m^{SPSA}$ \citep{malladi2023mezo, zhang2024private}, $M_m^{FD}$ \citep{duchioptzero}, and $M_m^{SP}$.
    \begin{enumerate}
        \item $M_m^{SPSA_\xi}(f, w)$: Sample i.i.d. $Z_1, ..., Z_m \sim \calN(0, I_d)$ and return 
        $\widehat{\nabla}^m_1 f_\xi(w) \define \frac{1}{m}\sum_{i=1}^m \frac{f(w + \xi {Z_i}) - f(w - \xi {Z_i})}{2\xi}{Z_i}$.
        \item $M_m^{FD_\xi}(f, w)$: Sample i.i.d. $Z_1, ..., Z_m \sim \calN(0, I_d)$ and return 
        $\widehat{\nabla}^m_2 f_\xi(w) \define \frac{1}{m}\sum_{i=1}^m \frac{f(w + \xi {Z_i}) - f(w)}{\xi}{Z_i}$.
        \item $M_m^{SP_\xi}(f, w)$: Sample i.i.d. ${Z_1, Z_2, ..., Z_m} \sim \calN(0, I_d)$ and return $\widehat{\nabla}^m_3 f_\xi(w) \define \frac{1}{m} \sum_{i=1}^m\frac{d}{\xi}f\left(w + \xi {Z_i}\right){Z_i}$
    \end{enumerate}
\end{example}

It is important to note that most work in ZO convex optimization relies on these estimators. Consequently, it becomes essential to define a subclass of randomized zeroth order oracles that generalizes this family of estimators and also \textit{provides a clearer understanding of their privacy implications}. Understanding the properties of such oracles is crucial for analyzing the privacy aspects of ZO methods that reveal their intermediate states.

\section{Privacy of Randomized Zeroth-Order Oracles}
\label{sec:privacyzosgd}

In this section, we define and discuss a subclass of randomized zeroth-order oracles as defined in Definition~\ref{def:zeroorderoracle}, and discuss privacy properties of this subclass. Let us start by defining \zpn oracles:

\begin{definition}
\label{def:multnoiseoracle}
    An oracle $\calO$ is a \zpn oracle if for any $f: \bb{R}^d \rightarrow \bb{R}$, it satisfies the following properties,
    \begin{enumerate}
        \item $\calO(f, w)$ returns $\mathbf{0}_d$ when $f(w) = 0$ for all $w \in \Real^d$ i.e. if $U \sim \calO(0, w)$ then $\bb{P}[U = \mathbf{0}_d] = 1$. 
        \item For $a > 0$, if $f(u) = \frac{a}{2}\twonorm{u}^2$, then $\calO(f, w)$ is a continuous probability measure for $w \neq \mathbf{0}_d$ i.e. if $U \sim \calO(f, w)$ and $w \neq \mathbf{0}_d$ then for all $c \in \bb{R}^d, \bb{P}[U = c] = 0$.
    \end{enumerate}
\end{definition}

It is important to note that, a \zpn oracle does not necessarily imply a zeroth order oracle. For instance, consider the estimator: Sample $Z \sim \calN(0, I_d)$ and return $\widehat{\nabla} f(u) = (Z^T\nabla f(u))Z$. We see that this estimator satisfies the property of being a \zpn oracle, even though it still uses first order information.

Nonetheless, \zpn oracles capture many popular ZO estimators used in the literature. In the following lemma, we show that several of the estimators discussed in Example~\ref{ex:elemestimators} satisfy Definition~\ref{def:multnoiseoracle}.

\vspace{0.1cm}

\begin{lemma}\label{estimator:multiplicative}
    $SPSA$, $FD$, $SP$, $M_m^{SPSA}$, $M_m^{FD}$ and $M_m^{SP}$ are \zpn oracles.
\end{lemma}

\vspace{0.1cm}

We defer the proof to \cref{app:estimatormultiplicative}. We also give example of another estimator proposed in \citet{zhang2024dpzero} and show that it is also a \zpn oracle. 

 The following lemma shows that any \zpn oracle is not differentially private.
\begin{theorem}
    \label{lemma:non-privacyofZOestimators}
    If $\calO$ is a \zpn oracle (as defined in Definition~\ref{def:multnoiseoracle}), then there exists an L-Lipschitz strongly convex loss function over the set $L\bb{B}^d$ and a pair of datasets such that $\calO$ is not $(\varepsilon, \delta)$-differentially private for any $\varepsilon < \infty$ and any $\delta < 1$.
\end{theorem}
\begin{proof}
Let $\calX = \{x_1, x_2, ..., x_n\}$ be the database with $x_i \in \Real^k$ and $\twonorm{x_i} \leq 1, \;\forall i$. Given this database, consider the  function 
    \begin{align*}
        \calL(w; \calX) = \frac{1}{n} \sum_{i = 1}^n \twonorm{x_i} \twonorm{w}^2,
    \end{align*}
    with parameter $w \in L\bb{B}^d$.
    Consider the neighboring databases $\calX = \{x_1, x_2, ..., x_{n-1}, x_n\}$ and $\calX' = \{x_1, x_2, ..., x_{n-1}, x_n'\}$ differing at the last entry (WLOG). Assigning $x_1, ..., x_{n-1}$ to be ${0}_k$. 
    For the last points take, $x_n = 0$ and $x_n' = \frac{1}{\sqrt{k}}{1}_k$. 
    With this construction, we have  $\calL(w; \calX) = 0$ and $\calL(w; \calX') = \frac{L}{n}\twonorm{w}^2$.
    Let $R_\calX \sim \calO(\calL(\cdot; \calX), w)$
    and $ R_{\calX'} \sim \calO(\calL(.; \calX'), w)$.
    By the property of a \zpn oracle, we have $\bb{P}\left[R_\calX = \mathbf{0}_d\right] = 1$, while $R_{\calX'}$ would be a continuous random variable. 
    Hence,  for  the singleton set $S = \{\mathbf{0}_s\}$, we have  $Pr[R_\calX \in S] = 1$ and $Pr[R_{\calX'} \in S] = 0$ because $R_{\calX'}$ is a continuous random variable with unbounded support. Clearly, $\bb{P}[R_\calX \in S] > e^\varepsilon \bb{P}[R_{\calX'} \in S] + \delta$  for any $\varepsilon < \infty$ and $\delta < 1$,  contradicting $(\varepsilon,\delta)$-differential privacy  definition.
\end{proof}

A few key observations about this result are
\begin{enumerate}
    \item \citet{abadipaper, charles2024finetuninglargelanguagemodels} showed that for GD and SGD algorithms to be differentially private, each parameter update must be differentially private if the attacker has access to the all the intermediate states. Thus, understanding the privacy preserving properties of distribution of updates (e.g. ZO oracles) helps us in understanding the privacy of the algorithm itself.
    \item Theorem~\ref{lemma:non-privacyofZOestimators} demonstrates that algorithms involving the sharing of gradient estimates or parameter updates sampled from \zpn oracles between parties are not private. For instance, in many federated learning mechanisms \citep{lowy2023privatefederatedlearningtrusted, pmlr-v206-lowy23a, fedother, gao2024privateheterogeneousfederatedlearning}, gradients (estimates) are shared from silos to a central server. In the absence of a trusted server, the lack of privacy of the updates, from a particular silo, poses a significant risk to the confidentiality of the data within the silo.
\end{enumerate}

It is also important to note that Theorem~\ref{lemma:non-privacyofZOestimators} does not dismiss the privacy guarantee of zeroth order estimators with \textit{independent} additive noise as discussed in \citet{zhang2024dpzero} and \citet{zhang2024private}. This is because gradient estimators with \textit{independent} additive noise do not satisfy the first property of Definition~\ref{def:multnoiseoracle}. Consider the gradient estimator with an additive noise to be $\widehat{\nabla} f(w) = \widehat{\nabla}' f(w) + \gamma$ where $\widehat{\nabla}' f(w): \Real^d \rightarrow \Real^d$ represents an arbitrary estimator and $\gamma \sim \calN(0, \sigma^2 I_d)$ is an independent noise for a $\sigma$ which satisfies $(\epsilon, \delta)$ DP (assuming that $\widehat{\nabla}' f(w)$ has bounded sensitivity). We see that (in the worst case) when $\widehat{\nabla}' f(w)$ is a constant $\widehat{\nabla} f(w)$ is still a continuous distribution since $\gamma$ is an independent noise added.

\section{Privacy of ZO SGD}
\label{sec:privacyzosgdunk}

At this point, we have given a partial answer to the question asked in \citet{zhang2024dpzero, tang2024privatefinetuninglargelanguage} with respect to privacy of \textit{their zeroth order oracles}. However, this result does not account for the fact when we have no knowledge about the intermediate states of the algorithm. \citet{ NEURIPS2021_7c6c1a7b, altschuler2022privacy} have proven that when one considers the case that the attacker has no access to hidden states, then after a small burn-in period, Projected Noisy SGD on strongly convex and convex functions incurs no additional loss in privacy as T increases. Thus, it is possible for some iterates to not be private individually, but the noise due to the zeroth order oracle ``accumulates" over time and gives certain privacy guarantees for the final iterate. Therefore, a natural question is as follows:

\fbox{
    \begin{minipage}{\columnwidth - 14pt}
    \textit{Is {the inherent noise of ZO Stochastic Gradient Descent (SGD) with a constant initialization sufficient to preserve privacy given access to the final iterate only?}  }
    \end{minipage}
}

The following theorem answers this question.

\begin{theorem}
    \label{multidim:fixedinit} Consider running T steps of Algorithm \ref{alg:proj_zo_sgd} using a \zpn oracle $\calO$, as defined in Definition~\ref{def:multnoiseoracle}, with $D > 0$ and $\calR_{init}$ as a fixed constant $w_0 \in \Real^d$ such that $\twonorm{w_0} < D$. Assume that the algorithm only returns the final iterate. Then, there exists an L-Lipschitz linear loss function over the set $[-D, D]^d \subset \Real^d$ such that for any $T \geq 1$ the output of Algorithm \ref{alg:proj_zo_sgd} is not $(\varepsilon, \delta)$-differentially private for any $\varepsilon < \infty$, $\delta < 1$
\end{theorem}
\begin{proof}
    Consider the following function for a database $\calX = \{x_1, x_2, ..., x_n\}$ where for all $i \in [n]$ $x_i \in \Real^k$ such that $\twonorm{x_i} \leq 1$ with parameter $w \in L\bb{B}^d$
    \begin{align*}
        \calL(w; \calX) = \frac{1}{n} \sum_{i = 1}^n \twonorm{x_i} \twonorm{w}^2.
    \end{align*}
    
    Consider the neighboring databases differing at the last entry (WLOG) $\calX = \{x_1, x_2, ..., x_{n-1}, x_n\}$ and $\calX' = \{x_1, x_2, ..., x_{n-1}, x_n'\}$. Assign $x_1, ..., x_{n-1}$ to be ${0}_k$. For the last points take, $x_n = 0$ and $x_n' = \frac{1}{\sqrt{k}}{1}_k$. With the above construction, we have the following two functions under $\calX$ and $\calX'$, $\calL(w; \calX) = 0$ and $\calL(w; \calX') = \frac{L}{n}\twonorm{w}^2$.
    
    Take the random variables for iterates corresponding to running Algorithm \ref{alg:proj_zo_sgd} on $\chi$ and $\chi'$ to be $W^\chi_t$ and $W^{\chi'}_t$ respectively.  Since $\calL(w;\chi) = 0$, then by the first property of $\calO$ along with the fact that $\calO(f, w) = \mathbf{0}_d$ for a constant $f$, we have that $\bb{P}\left[\calO(\calL(\cdot, \chi), w) = \mathbf{0}_d\right] = 1$. This implies that the iterate would not change on each step and since $w_0 \in D\Ball^d$, projection would be identity at each step, which implies $W^\chi_t = w_0$.
    
    On the other hand, $W_t^{\chi'} = \Pi_{D\Ball^d}\paren{W_{t-1}^{\chi'} - \eta \widehat{\nabla}\calL(W_{t-1}; \calX')}$. Since the minima of $\calL(w; \calX')$ lies strictly at $w = {0}$, it implies that $\calL(w_0; \calX') \neq 0$ and since $\calL(w; \calX')$ is not a constant, then by the second property in Definition~\ref{def:multnoiseoracle} it implies that $\widehat{\nabla}\calL(W^{\chi'}_{t-1}, \chi')$ is a continuous distribution, which implies that $W_{t-1}^{\chi'} - \eta \widehat{\nabla}\calL(W_{t-1}; \calX')$ is a continuous distribution. However, since we are projecting into a ball at every iteration, then it implies that for $w \in D\bb{B}^d$ such that $\twonorm{w} < D$, $\bb{P}[W_t^{\chi'} = w] = \bb{P}[W_{t-1}^{\chi'} - \eta \widehat{\nabla}\calL(W_{t-1}; \calX') = w]= 0$ while for $\twonorm{w} = D$, $\bb{P}[W_t^{\chi'} = D] = \bb{P}[W_{t-1}^{\chi'} - \eta \widehat{\nabla}\calL(W_{t-1}; \calX') \geq D]$.      
    Thus, if we consider the singleton set $S = \{w_0\}$, then we would get that $Pr[W^\chi_T \in S] = 1$ and $Pr[W^{\chi'}_T \in S] = 0$ because $\twonorm{w_0}$ is strictly less than $D$. Hence, $\bb{P}[W^\chi_T \in S] > e^\varepsilon \bb{P}[W^\chi_T \in S] + \delta$ for any $\varepsilon < \infty$ and $\delta < 1$.
\end{proof}

 This result answers the open question proposed by \citet{tang2024privatefinetuninglargelanguage} and \citet{zhang2024dpzero} on the privacy of zeroth order (S)GD with a fixed initialization. Thus, it dismisses the promise of privacy of (S)GD under such oracles. It is important to note that our framework does not capture certain algorithms like Stochastic Zeroth Order Conditional Gradient Descent \citep{Balasubramanian2022-zu} or Mirror Descent \citep{duchioptzero}. It would be interesting to see if such a result can be generalized for other classes of optimization algorithms discussed in \citet{duchioptzero} or \citet{Balasubramanian2022-zu}.

Notably, in modern machine learning tasks, model parameters are initialized randomly \citep{7410480}. Thus, if we consider the optimization of the functions corresponding to two neighbouring datasets as defined in Definition~\ref{def:dp}, then we cannot comment about the distribution of the final iterate. Therefore, a natural extension to our previous question emerges as follows:

\fbox{
    \begin{minipage}{\columnwidth - 14pt}
    \textit{Is {the inherent noise of ZO Stochastic Gradient Descent (SGD) \textbf{with random initialization} sufficient to preserve privacy given access to the final iterate only?}  }
    \end{minipage}
}

 It is important to note that the distribution of the initial iterate is continuous. Thus, by proof of Theorem~\ref{multidim:fixedinit}, the continuity of the \zpn oracle is no longer sufficient to ensure privacy. We define a new class of oracles below for which we analyse privacy.

\begin{definition}
    \label{def:badoracle}
    An update oracle $\calO$ is $C_s$-\ac (AC) if there exists an index $i^* \in [d]$ such that if we consider the function class $\c{F} = \left\{\inner{g\mathbf{e}_{i^*}}{w}: - L \leq g \leq 0\right\}$, then for any $h \in \c{F}$ and $\overline{w} \in \Real^d$, if $U \sim \calO(h, \overline{w})$ then,
    \begin{enumerate}
        \item $\{\nabla h(\overline{w})\}_{i^*} = 0$ implies $U = \mathbf{0}_d$ w.p. 1 i.e. $\bb{P}\left[U = \mathbf{0}_d\right] = 1$
        \item $\{\nabla h(\overline{w})\}_{i^*} \neq 0$ implies $\bb{P}\left[\{U\}_{i^*} < 0\right] = 1$ 
        \item $\bb{E}\left[\{U\}_{i^*}\right] \leq \{\nabla h(\overline{w})\}_{i^*}$
        \item For any set of $\{w_1, w_2, ..., w_N\} \subset \Real^d$, let $U_j \sim \calO(h, w_j)$ independently for all $j \in [N]$. Then $\Ex{\paren{\sum_{j=1}^N \{U_j\}_{i^*}}^2}\leq C_s\Ex{\sum_{j=1}^N \{U_j\}_{i^*}}^2$
    \end{enumerate}
\end{definition}

This oracle roughly gives us a stronger guarantee (than \zpn oracles) that updates for a certain class of functions (with non-zero gradients) would be strictly increasing while ensuring a zero-preserving property (like \zpn oracles) on functions which are zero over $\Real^d$. Further its distributional properties, like the condition on expectation and variance, are necessary for ensuring an ``anti-concenteration'' phenomena over the updates (hence the name). This suggests that the updates from AC oracles would shift the iterates away from the initial point with a good probability, even when it is randomized.

Similar to the definition of \zpn oracles, this seemingly specific oracle class is able to capture the two point oracles we have discussed so far. Specifically, we show that the $SPSA$ and $FD$ estimators are 3-\acs~oracles.

\begin{lemma}
    SPSA and FD are 3-\acs.
\end{lemma}

\begin{proof} We defer the complete proof to Appendix~\ref{app:baddiscussion}. (Sketch) The proof comes from fixing $i^* = 1$ and analyzing the output of the estimator for any function in the function class ($\calF$) corresponding to $i^*$  as specified in Definition~\ref{def:badoracle}. Due to the difference form of the estimators, the distribution of the first index of the outputs of these estimators turns out to be a chi-squared random variable with the degree of freedom 1 scaled by the parameter $g$ defined in Definition~\ref{def:badoracle}. Hence, this distribution satisfies all the properties mentioned in Definition~~\ref{def:badoracle}, thus completing the proof. 
\end{proof}

 We further prove that the two-point oracle proposed by \citet{duchioptzero} is also 3-\acs. Moreover, the property of being an AC oracle is preserved under the mean extension, as shown by the following lemma.

\begin{lemma}
    If an oracle $\c{E}$ is $C_s$-\acs~ then $M_m^{\c{E}}$ is $C_s$-\acs.
\end{lemma}
\begin{proof}
    We defer the complete proof to Appendix~\ref{app:baddiscussion}. (Sketch) The argument here is, we are taking mean of i.i.d. random variables $U^{(f,w_j)}_i \sim \c{E}(f, w_j)$ for all $i \in [m]$ for $\c{E}$ which are $C_s$ anti-concentrated. Hence, the ``strict'' properties (namely zero-preserving property and the strict negativity under non-zero gradient of AC oracles) of samples of $\c{E}(f, w_j)$ also comply to the mean of i.i.d. samples from distribution satisfying property 1 and property 2. Due to linearity of expectation, the same upper bound on the expectation holds for the mean satisfying property 3. Due to reordering of RV and Young's inequality, property 4 also complies to the mean.
\end{proof}

This lemma directly implies that $M_m^{FD}$ and $M_m^{SPSA}$ are 3-\acs. 

If we restrict ourselves to the class of zeroth order oracles, then we make an interesting observation: All the estimators discussed above that use at least \textit{two function evaluations} satisfy the properties mentioned in Definition~\ref{def:badoracle}. On the other hand, for the $SP$ estimator defined in Example~\ref{ex:elemestimators}, we observe that it does not satisfy the third property of Definition~\ref{def:badoracle}. Thus, restricted to zeroth order oracles, the idea of an \acs~oracle roughly captures an important distinction between popular single-point and two-point estimator(s).

Using the definition of an \acs~oracle, we show that even with unknown (but random) initialization, the model loses privacy for a large enough diameter of the constraint set.

\begin{theorem}
    \label{multidim:varinit}
     Consider running T steps of Algorithm \ref{alg:proj_zo_sgd} using a $C_s$-\acs~oracle $\calO$, as defined in Definition~\ref{def:badoracle}, with $D > \frac{\eta L}{2n}$ and $\calR_{init}$ is $\calN(0, \sigma^2I)$. Assume that the algorithm only returns the final iterate. Then, there exists an L-Lipschitz linear loss function over the set $[-D, D]^d \subset \Real^d$ such that for any $T \geq 1$ the output of Algorithm \ref{alg:proj_zo_sgd} is not $(\varepsilon, \delta)$-differentially private for any $\varepsilon, \delta$ satisfying $\delta \leq \frac{1}{16C_s}\max\left\{\frac{1}{T^{2/3}}, \frac{\eta L}{2nD}\right\}$ and
    \begin{align*}
        \varepsilon \leq \min\left\{\frac{\eta^2L^2T^{4/3}}{8n^2\sigma^2}, \frac{D^2}{2\sigma^2}\right\} + \ln\paren{\frac{\sqrt{2\pi}L\eta}{64C_sn\sigma}}.
    \end{align*}
\end{theorem}

We defer the proof to Appendix~\ref{app:multidim:varinit}. %
If we take the values of $n \in [100, 1000]$, $\eta \in [0.001, 0.01]$, $C_s = 3$, $L = 100$, and $\sigma = \frac{1}{1020}$, we see that for a large enough $D\left(\geq \frac{T^{2/3}}{200}\right)$ this result gives us that privacy is not possible (roughly) for $\epsilon \leq 896 \cdot T^{4/3}$ and $\delta \leq \frac{1}{96T^{2/3}}$. Notably, we get the same flavor of result as obtained in the lower bound construction of \citet{altschuler2022privacy}. Hence, it would be interesting to see if a matching upper bound is attainable for Projected ZO-GD.

\section{Conclusion and Open Problems}

In this work, we demonstrated that projected (stochastic) ZO-gradient descent cannot ensure differential privacy without incorporating additional (additive) noise. While our findings apply to a broad class of zeroth-order oracles, the algorithmic framework we used is limited to the class of projected SGD. Potential future directions for this problem include:

\begin{enumerate}
    \item Can we prove Theorem~\ref{multidim:varinit} for other function classes (e.g., strongly convex loss functions)?
    \item Can we extend our (lower-bound) analysis beyond projected (stochastic) ZO-GD?
    \item Although zeroth-order estimators do not offer inherent privacy on their own, can they amplify the privacy of other (additive) noisy methods?
\end{enumerate}
The answer to any of these questions will help understand the limitations and potential of zeroth-order optimization for private optimization.

\section*{Acknowledgments}
The work of MR and DG was supported by a gift from Meta, a gift from Google, and a gift from Amazon. VS was supported by NSF CAREER Award CCF-2239265 and
an Amazon Research Award. The work was partly done when VS was visiting the Simons Institute in Berkeley.

\bibliographystyle{unsrtnat}
\bibliography{sample}

\begin{thebibliography}{50}
\providecommand{\natexlab}[1]{#1}
\providecommand{\url}[1]{\texttt{#1}}
\expandafter\ifx\csname urlstyle\endcsname\relax
  \providecommand{\doi}[1]{doi: #1}\else
  \providecommand{\doi}{doi: \begingroup \urlstyle{rm}\Url}\fi

\bibitem[Tang et~al.(2024)Tang, Panda, Nasr, Mahloujifar, and Mittal]{tang2024privatefinetuninglargelanguage}
Xinyu Tang, Ashwinee Panda, Milad Nasr, Saeed Mahloujifar, and Prateek Mittal.
\newblock Private fine-tuning of large language models with zeroth-order optimization, 2024.
\newblock URL \url{https://arxiv.org/abs/2401.04343}.

\bibitem[Zhang et~al.(2024{\natexlab{a}})Zhang, Li, Thekumparampil, Oh, and He]{zhang2024dpzero}
Liang Zhang, Bingcong Li, Kiran~Koshy Thekumparampil, Sewoong Oh, and Niao He.
\newblock {DPZero}: Private fine-tuning of language models without backpropagation.
\newblock In \emph{International Conference on Machine Learning}. PMLR, 2024{\natexlab{a}}.

\bibitem[Malladi et~al.(2023)Malladi, Gao, Nichani, Damian, Lee, Chen, and Arora]{malladi2023mezo}
Sadhika Malladi, Tianyu Gao, Eshaan Nichani, Alex Damian, Jason~D Lee, Danqi Chen, and Sanjeev Arora.
\newblock Fine-tuning language models with just forward passes.
\newblock \emph{Advances in Neural Information Processing Systems}, 36:\penalty0 53038--53075, 2023.

\bibitem[Spall(1992)]{spsa}
J.C. Spall.
\newblock Multivariate stochastic approximation using a simultaneous perturbation gradient approximation.
\newblock \emph{IEEE Transactions on Automatic Control}, 37\penalty0 (3):\penalty0 332--341, 1992.
\newblock \doi{10.1109/9.119632}.

\bibitem[Zhang et~al.(2024{\natexlab{b}})Zhang, Li, Hong, Li, Zhang, Zheng, Chen, Lee, Yin, Hong, Wang, Liu, and Chen]{zhang2024revisitingzerothorderoptimizationmemoryefficient}
Yihua Zhang, Pingzhi Li, Junyuan Hong, Jiaxiang Li, Yimeng Zhang, Wenqing Zheng, Pin-Yu Chen, Jason~D. Lee, Wotao Yin, Mingyi Hong, Zhangyang Wang, Sijia Liu, and Tianlong Chen.
\newblock Revisiting zeroth-order optimization for memory-efficient llm fine-tuning: A benchmark, 2024{\natexlab{b}}.
\newblock URL \url{https://arxiv.org/abs/2402.11592}.

\bibitem[Guo et~al.(2024)Guo, Long, Zeng, Liu, Yang, Ran, Gardner, Bastani, Sa, Yu, Chen, and Xu]{guo2024zerothorderfinetuningllmsextreme}
Wentao Guo, Jikai Long, Yimeng Zeng, Zirui Liu, Xinyu Yang, Yide Ran, Jacob~R. Gardner, Osbert Bastani, Christopher~De Sa, Xiaodong Yu, Beidi Chen, and Zhaozhuo Xu.
\newblock Zeroth-order fine-tuning of llms with extreme sparsity, 2024.
\newblock URL \url{https://arxiv.org/abs/2406.02913}.

\bibitem[Dwork et~al.(2010{\natexlab{a}})Dwork, Rothblum, and Vadhan]{DworkRothblumVadhan2010}
Cynthia Dwork, Guy~N. Rothblum, and Salil Vadhan.
\newblock Boosting and differential privacy.
\newblock In \emph{Proceedings of the 2010 IEEE 51st Annual Symposium on Foundations of Computer Science}, FOCS '10, page 51–60, USA, 2010{\natexlab{a}}. IEEE Computer Society.
\newblock ISBN 9780769542447.
\newblock \doi{10.1109/FOCS.2010.12}.
\newblock URL \url{https://doi.org/10.1109/FOCS.2010.12}.

\bibitem[Dwork et~al.(2010{\natexlab{b}})Dwork, Rothblum, and Vadhan]{boosting}
Cynthia Dwork, Guy~N. Rothblum, and Salil Vadhan.
\newblock Boosting and differential privacy.
\newblock In \emph{2010 IEEE 51st Annual Symposium on Foundations of Computer Science}, pages 51--60, 2010{\natexlab{b}}.
\newblock \doi{10.1109/FOCS.2010.12}.

\bibitem[Chaudhuri and Hsu(2011)]{pmlr-v19-chaudhuri11a}
Kamalika Chaudhuri and Daniel Hsu.
\newblock Sample complexity bounds for differentially private learning.
\newblock In Sham~M. Kakade and Ulrike von Luxburg, editors, \emph{Proceedings of the 24th Annual Conference on Learning Theory}, volume~19 of \emph{Proceedings of Machine Learning Research}, pages 155--186, Budapest, Hungary, 09--11 Jun 2011. PMLR.
\newblock URL \url{https://proceedings.mlr.press/v19/chaudhuri11a.html}.

\bibitem[Cohen et~al.(2023)Cohen, Lyu, Nelson, Sarl\'{o}s, and Stemmer]{thresholdfunctions}
Edith Cohen, Xin Lyu, Jelani Nelson, Tam\'{a}s Sarl\'{o}s, and Uri Stemmer.
\newblock Optimal differentially private learning of thresholds and quasi-concave optimization.
\newblock In \emph{Proceedings of the 55th Annual ACM Symposium on Theory of Computing}, STOC 2023, page 472–482, New York, NY, USA, 2023. Association for Computing Machinery.
\newblock ISBN 9781450399135.
\newblock \doi{10.1145/3564246.3585148}.
\newblock URL \url{https://doi.org/10.1145/3564246.3585148}.

\bibitem[Gonen et~al.(2019)Gonen, Hazan, and Moran]{NEURIPS2019_700fdb2b}
Alon Gonen, Elad Hazan, and Shay Moran.
\newblock Private learning implies online learning: An efficient reduction.
\newblock In H.~Wallach, H.~Larochelle, A.~Beygelzimer, F.~d\textquotesingle Alch\'{e}-Buc, E.~Fox, and R.~Garnett, editors, \emph{Advances in Neural Information Processing Systems}, volume~32. Curran Associates, Inc., 2019.
\newblock URL \url{https://proceedings.neurips.cc/paper_files/paper/2019/file/700fdb2ba62d4554dc268c65add4b16e-Paper.pdf}.

\bibitem[Block et~al.(2024)Block, Bun, Desai, Shetty, and Wu]{block2024oracleefficientdifferentiallyprivatelearning}
Adam Block, Mark Bun, Rathin Desai, Abhishek Shetty, and Steven Wu.
\newblock Oracle-efficient differentially private learning with public data, 2024.
\newblock URL \url{https://arxiv.org/abs/2402.09483}.

\bibitem[Chaudhuri et~al.(2011)Chaudhuri, Monteleoni, and Sarwate]{JMLR:v12:chaudhuri11a}
Kamalika Chaudhuri, Claire Monteleoni, and Anand~D. Sarwate.
\newblock Differentially private empirical risk minimization.
\newblock \emph{Journal of Machine Learning Research}, 12\penalty0 (29):\penalty0 1069--1109, 2011.
\newblock URL \url{http://jmlr.org/papers/v12/chaudhuri11a.html}.

\bibitem[Bassily et~al.(2014)Bassily, Smith, and Thakurta]{BassilySmith2014}
Raef Bassily, Adam Smith, and Abhradeep Thakurta.
\newblock Private empirical risk minimization: Efficient algorithms and tight error bounds.
\newblock In \emph{2014 IEEE 55th Annual Symposium on Foundations of Computer Science}, pages 464--473, 2014.
\newblock \doi{10.1109/FOCS.2014.56}.

\bibitem[Lowy et~al.(2024)Lowy, Ullman, and Wright]{pmlr-v235-lowy24b}
Andrew Lowy, Jonathan Ullman, and Stephen Wright.
\newblock How to make the gradients small privately: Improved rates for differentially private non-convex optimization.
\newblock In Ruslan Salakhutdinov, Zico Kolter, Katherine Heller, Adrian Weller, Nuria Oliver, Jonathan Scarlett, and Felix Berkenkamp, editors, \emph{Proceedings of the 41st International Conference on Machine Learning}, volume 235 of \emph{Proceedings of Machine Learning Research}, pages 32904--32923. PMLR, 21--27 Jul 2024.
\newblock URL \url{https://proceedings.mlr.press/v235/lowy24b.html}.

\bibitem[Arora et~al.(2023)Arora, Bassily, Gonz\'{a}lez, Guzm\'{a}n, Menart, and Ullah]{aroraprivacy}
Raman Arora, Raef Bassily, Tom\'{a}s Gonz\'{a}lez, Crist\'{o}bal Guzm\'{a}n, Michael Menart, and Enayat Ullah.
\newblock Faster rates of convergence to stationary points in differentially private optimization.
\newblock In \emph{Proceedings of the 40th International Conference on Machine Learning}, ICML'23. JMLR.org, 2023.

\bibitem[Feldman et~al.(2020)Feldman, Koren, and Talwar]{optlinear}
Vitaly Feldman, Tomer Koren, and Kunal Talwar.
\newblock Private stochastic convex optimization: optimal rates in linear time.
\newblock In \emph{Proceedings of the 52nd Annual ACM SIGACT Symposium on Theory of Computing}, STOC 2020, page 439–449, New York, NY, USA, 2020. Association for Computing Machinery.
\newblock ISBN 9781450369794.
\newblock \doi{10.1145/3357713.3384335}.
\newblock URL \url{https://doi.org/10.1145/3357713.3384335}.

\bibitem[Bassily et~al.(2019)Bassily, Feldman, Talwar, and Guha~Thakurta]{bassilyoptimal}
Raef Bassily, Vitaly Feldman, Kunal Talwar, and Abhradeep Guha~Thakurta.
\newblock Private stochastic convex optimization with optimal rates.
\newblock \emph{Advances in neural information processing systems}, 32, 2019.

\bibitem[Gopi et~al.(2022)Gopi, Lee, and Liu]{pmlr-v178-gopi22a}
Sivakanth Gopi, Yin~Tat Lee, and Daogao Liu.
\newblock Private convex optimization via exponential mechanism.
\newblock In Po-Ling Loh and Maxim Raginsky, editors, \emph{Proceedings of Thirty Fifth Conference on Learning Theory}, volume 178 of \emph{Proceedings of Machine Learning Research}, pages 1948--1989. PMLR, 02--05 Jul 2022.
\newblock URL \url{https://proceedings.mlr.press/v178/gopi22a.html}.

\bibitem[Altschuler and Talwar(2022)]{altschuler2022privacy}
Jason Altschuler and Kunal Talwar.
\newblock Privacy of noisy stochastic gradient descent: More iterations without more privacy loss.
\newblock In Alice~H. Oh, Alekh Agarwal, Danielle Belgrave, and Kyunghyun Cho, editors, \emph{Advances in Neural Information Processing Systems}, 2022.
\newblock URL \url{https://openreview.net/forum?id=pDUYkwrx__w}.

\bibitem[Carmon et~al.(2023)Carmon, Jambulapati, Jin, Lee, Liu, Sidford, and Tian]{resque}
Yair Carmon, Arun Jambulapati, Yujia Jin, Yin~Tat Lee, Daogao Liu, Aaron Sidford, and Kevin Tian.
\newblock Resqueing parallel and private stochastic convex optimization.
\newblock In \emph{2023 IEEE 64th Annual Symposium on Foundations of Computer Science (FOCS)}, pages 2031--2058, 2023.
\newblock \doi{10.1109/FOCS57990.2023.00124}.

\bibitem[Gopi et~al.(2023)Gopi, Lee, Liu, Shen, and Tian]{generalnorms}
Sivakanth Gopi, Yin~Tat Lee, Daogao Liu, Ruoqi Shen, and Kevin Tian.
\newblock Private convex optimization in general norms.
\newblock In \emph{Proceedings of the 2023 Annual ACM-SIAM Symposium on Discrete Algorithms (SODA)}, pages 5068--5089. SIAM, 2023.

\bibitem[Abadi et~al.(2016)Abadi, Chu, Goodfellow, McMahan, Mironov, Talwar, and Zhang]{abadipaper}
Martin Abadi, Andy Chu, Ian Goodfellow, H.~Brendan McMahan, Ilya Mironov, Kunal Talwar, and Li~Zhang.
\newblock Deep learning with differential privacy.
\newblock In \emph{Proceedings of the 2016 ACM SIGSAC Conference on Computer and Communications Security}, CCS '16, page 308–318, New York, NY, USA, 2016. Association for Computing Machinery.
\newblock ISBN 9781450341394.
\newblock \doi{10.1145/2976749.2978318}.
\newblock URL \url{https://doi.org/10.1145/2976749.2978318}.

\bibitem[Yousefpour et~al.(2022)Yousefpour, Shilov, Sablayrolles, Testuggine, Prasad, Malek, Nguyen, Ghosh, Bharadwaj, Zhao, Cormode, and Mironov]{opacus}
Ashkan Yousefpour, Igor Shilov, Alexandre Sablayrolles, Davide Testuggine, Karthik Prasad, Mani Malek, John Nguyen, Sayan Ghosh, Akash Bharadwaj, Jessica Zhao, Graham Cormode, and Ilya Mironov.
\newblock Opacus: User-friendly differential privacy library in pytorch, 2022.
\newblock URL \url{https://arxiv.org/abs/2109.12298}.

\bibitem[Erlingsson et~al.(2014)Erlingsson, Pihur, and Korolova]{10.1145/2660267.2660348}
\'{U}lfar Erlingsson, Vasyl Pihur, and Aleksandra Korolova.
\newblock Rappor: Randomized aggregatable privacy-preserving ordinal response.
\newblock In \emph{Proceedings of the 2014 ACM SIGSAC Conference on Computer and Communications Security}, CCS '14, page 1054–1067, New York, NY, USA, 2014. Association for Computing Machinery.
\newblock ISBN 9781450329576.
\newblock \doi{10.1145/2660267.2660348}.
\newblock URL \url{https://doi.org/10.1145/2660267.2660348}.

\bibitem[Ding et~al.(2017)Ding, Kulkarni, and Yekhanin]{NIPS2017_253614bb}
Bolin Ding, Janardhan Kulkarni, and Sergey Yekhanin.
\newblock Collecting telemetry data privately.
\newblock In I.~Guyon, U.~Von Luxburg, S.~Bengio, H.~Wallach, R.~Fergus, S.~Vishwanathan, and R.~Garnett, editors, \emph{Advances in Neural Information Processing Systems}, volume~30. Curran Associates, Inc., 2017.
\newblock URL \url{https://proceedings.neurips.cc/paper_files/paper/2017/file/253614bbac999b38b5b60cae531c4969-Paper.pdf}.

\bibitem[Rogers et~al.(2021)Rogers, Subramaniam, Peng, Durfee, Lee, Kancha, Sahay, and Ahammad]{Rogers_Subramaniam_Peng_Durfee_Lee_Kancha_Sahay_Ahammad_2021}
Ryan Rogers, Subbu Subramaniam, Sean Peng, David Durfee, Seunghyun Lee, Santosh~Kumar Kancha, Shraddha Sahay, and Parvez Ahammad.
\newblock Linkedin’s audience engagements api: A privacy preserving data analytics system at scale.
\newblock \emph{Journal of Privacy and Confidentiality}, 11\penalty0 (3), Dec. 2021.
\newblock \doi{10.29012/jpc.782}.
\newblock URL \url{https://journalprivacyconfidentiality.org/index.php/jpc/article/view/782}.

\bibitem[Wibisono et~al.(2012)Wibisono, Wainwright, Jordan, and Duchi]{NIPS2012_e555ebe0}
Andre Wibisono, Martin~J Wainwright, Michael Jordan, and John~C Duchi.
\newblock Finite sample convergence rates of zero-order stochastic optimization methods.
\newblock In F.~Pereira, C.J. Burges, L.~Bottou, and K.Q. Weinberger, editors, \emph{Advances in Neural Information Processing Systems}, volume~25. Curran Associates, Inc., 2012.
\newblock URL \url{https://proceedings.neurips.cc/paper_files/paper/2012/file/e555ebe0ce426f7f9b2bef0706315e0c-Paper.pdf}.

\bibitem[Nesterov and Spokoiny(2015)]{Nesterov2015RandomGM}
Yurii Nesterov and Vladimir~G. Spokoiny.
\newblock Random gradient-free minimization of convex functions.
\newblock \emph{Foundations of Computational Mathematics}, 17:\penalty0 527 -- 566, 2015.
\newblock URL \url{https://api.semanticscholar.org/CorpusID:2147817}.

\bibitem[Shamir(2013)]{pmlr-v28-shamir13}
Ohad Shamir.
\newblock On the complexity of bandit and derivative-free stochastic convex optimization.
\newblock In Shai Shalev-Shwartz and Ingo Steinwart, editors, \emph{Proceedings of the 26th Annual Conference on Learning Theory}, volume~30 of \emph{Proceedings of Machine Learning Research}, pages 3--24, Princeton, NJ, USA, 12--14 Jun 2013. PMLR.
\newblock URL \url{https://proceedings.mlr.press/v30/Shamir13.html}.

\bibitem[Duchi et~al.(2015)Duchi, Jordan, Wainwright, and Wibisono]{duchioptzero}
John~C. Duchi, Michael~I. Jordan, Martin~J. Wainwright, and Andre Wibisono.
\newblock Optimal rates for zero-order convex optimization: The power of two function evaluations.
\newblock \emph{IEEE Transactions on Information Theory}, 61\penalty0 (5):\penalty0 2788--2806, 2015.
\newblock \doi{10.1109/TIT.2015.2409256}.

\bibitem[Yue et~al.(2023)Yue, Yang, Fang, and Lin]{pmlr-v195-yue23b}
Pengyun Yue, Long Yang, Cong Fang, and Zhouchen Lin.
\newblock Zeroth-order optimization with weak dimension dependency.
\newblock In Gergely Neu and Lorenzo Rosasco, editors, \emph{Proceedings of Thirty Sixth Conference on Learning Theory}, volume 195 of \emph{Proceedings of Machine Learning Research}, pages 4429--4472. PMLR, 12--15 Jul 2023.
\newblock URL \url{https://proceedings.mlr.press/v195/yue23b.html}.

\bibitem[Hu et~al.(2022)Hu, yelong shen, Wallis, Allen-Zhu, Li, Wang, Wang, and Chen]{hu2022lora}
Edward~J Hu, yelong shen, Phillip Wallis, Zeyuan Allen-Zhu, Yuanzhi Li, Shean Wang, Lu~Wang, and Weizhu Chen.
\newblock Lo{RA}: Low-rank adaptation of large language models.
\newblock In \emph{International Conference on Learning Representations}, 2022.
\newblock URL \url{https://openreview.net/forum?id=nZeVKeeFYf9}.

\bibitem[Kulkarni et~al.(2021)Kulkarni, Lee, and Liu]{NEURIPS2021_211c1e0b}
Janardhan Kulkarni, Yin~Tat Lee, and Daogao Liu.
\newblock Private non-smooth erm and sco in subquadratic steps.
\newblock In M.~Ranzato, A.~Beygelzimer, Y.~Dauphin, P.S. Liang, and J.~Wortman Vaughan, editors, \emph{Advances in Neural Information Processing Systems}, volume~34, pages 4053--4064. Curran Associates, Inc., 2021.
\newblock URL \url{https://proceedings.neurips.cc/paper_files/paper/2021/file/211c1e0b83b9c69fa9c4bdede203c1e3-Paper.pdf}.

\bibitem[Asi et~al.(2021)Asi, Feldman, Koren, and Talwar]{pmlr-v139-asi21b}
Hilal Asi, Vitaly Feldman, Tomer Koren, and Kunal Talwar.
\newblock Private stochastic convex optimization: Optimal rates in l1 geometry.
\newblock In Marina Meila and Tong Zhang, editors, \emph{Proceedings of the 38th International Conference on Machine Learning}, volume 139 of \emph{Proceedings of Machine Learning Research}, pages 393--403. PMLR, 18--24 Jul 2021.
\newblock URL \url{https://proceedings.mlr.press/v139/asi21b.html}.

\bibitem[Zhang et~al.(2024{\natexlab{c}})Zhang, Tran, and Cutkosky]{zhang2024private}
Qinzi Zhang, Hoang Tran, and Ashok Cutkosky.
\newblock Private zeroth-order nonsmooth nonconvex optimization.
\newblock In \emph{The Twelfth International Conference on Learning Representations}, 2024{\natexlab{c}}.
\newblock URL \url{https://openreview.net/forum?id=IzqZbNMZ0M}.

\bibitem[Shokri et~al.(2017)Shokri, Stronati, Song, and Shmatikov]{shokri2017membership}
Reza Shokri, Marco Stronati, Congzheng Song, and Vitaly Shmatikov.
\newblock Membership inference attacks against machine learning models.
\newblock In \emph{2017 IEEE symposium on security and privacy (SP)}, pages 3--18. IEEE, 2017.

\bibitem[Dwork et~al.(2006)Dwork, McSherry, Nissim, and Smith]{originalDP}
Cynthia Dwork, Frank McSherry, Kobbi Nissim, and Adam Smith.
\newblock Calibrating noise to sensitivity in private data analysis.
\newblock In Shai Halevi and Tal Rabin, editors, \emph{Theory of Cryptography}, pages 265--284, Berlin, Heidelberg, 2006. Springer Berlin Heidelberg.
\newblock ISBN 978-3-540-32732-5.

\bibitem[Flaxman et~al.(2005)Flaxman, Kalai, and McMahan]{singlepoint}
Abraham~D. Flaxman, Adam~Tauman Kalai, and H.~Brendan McMahan.
\newblock Online convex optimization in the bandit setting: gradient descent without a gradient.
\newblock In \emph{Proceedings of the Sixteenth Annual ACM-SIAM Symposium on Discrete Algorithms}, SODA '05, page 385–394, USA, 2005. Society for Industrial and Applied Mathematics.
\newblock ISBN 0898715857.

\bibitem[Nemirovsky and Yudin(1983)]{nemrosyud}
A.~S. Nemirovsky and D.~B. Yudin.
\newblock Problem complexity and method efficiency in optimization.
\newblock \emph{Wiley}, 1983.

\bibitem[Charles et~al.(2024)Charles, Ganesh, McKenna, McMahan, Mitchell, Pillutla, and Rush]{charles2024finetuninglargelanguagemodels}
Zachary Charles, Arun Ganesh, Ryan McKenna, H.~Brendan McMahan, Nicole Mitchell, Krishna Pillutla, and Keith Rush.
\newblock Fine-tuning large language models with user-level differential privacy, 2024.
\newblock URL \url{https://arxiv.org/abs/2407.07737}.

\bibitem[Lowy and Razaviyayn(2023)]{lowy2023privatefederatedlearningtrusted}
Andrew Lowy and Meisam Razaviyayn.
\newblock Private federated learning without a trusted server: Optimal algorithms for convex losses, 2023.
\newblock URL \url{https://arxiv.org/abs/2106.09779}.

\bibitem[Lowy et~al.(2023)Lowy, Ghafelebashi, and Razaviyayn]{pmlr-v206-lowy23a}
Andrew Lowy, Ali Ghafelebashi, and Meisam Razaviyayn.
\newblock Private non-convex federated learning without a trusted server.
\newblock In Francisco Ruiz, Jennifer Dy, and Jan-Willem van~de Meent, editors, \emph{Proceedings of The 26th International Conference on Artificial Intelligence and Statistics}, volume 206 of \emph{Proceedings of Machine Learning Research}, pages 5749--5786. PMLR, 25--27 Apr 2023.
\newblock URL \url{https://proceedings.mlr.press/v206/lowy23a.html}.

\bibitem[Aldaghri et~al.(2023)Aldaghri, Mahdavifar, and Beirami]{fedother}
Nasser Aldaghri, Hessam Mahdavifar, and Ahmad Beirami.
\newblock Federated learning with heterogeneous differential privacy, 2023.
\newblock URL \url{https://arxiv.org/abs/2110.15252}.

\bibitem[Gao et~al.(2024)Gao, Lowy, Zhou, and Wright]{gao2024privateheterogeneousfederatedlearning}
Changyu Gao, Andrew Lowy, Xingyu Zhou, and Stephen~J. Wright.
\newblock Private heterogeneous federated learning without a trusted server revisited: Error-optimal and communication-efficient algorithms for convex losses, 2024.
\newblock URL \url{https://arxiv.org/abs/2407.09690}.

\bibitem[Chourasia et~al.(2021)Chourasia, Ye, and Shokri]{NEURIPS2021_7c6c1a7b}
Rishav Chourasia, Jiayuan Ye, and Reza Shokri.
\newblock Differential privacy dynamics of langevin diffusion and noisy gradient descent.
\newblock In M.~Ranzato, A.~Beygelzimer, Y.~Dauphin, P.S. Liang, and J.~Wortman Vaughan, editors, \emph{Advances in Neural Information Processing Systems}, volume~34, pages 14771--14781. Curran Associates, Inc., 2021.
\newblock URL \url{https://proceedings.neurips.cc/paper_files/paper/2021/file/7c6c1a7bfde175bed616b39247ccace1-Paper.pdf}.

\bibitem[Balasubramanian and Ghadimi(2022)]{Balasubramanian2022-zu}
Krishnakumar Balasubramanian and Saeed Ghadimi.
\newblock {Zeroth-Order} nonconvex stochastic optimization: Handling constraints, high dimensionality, and saddle points.
\newblock \emph{Foundations of Computational Mathematics}, 22\penalty0 (1):\penalty0 35--76, February 2022.

\bibitem[He et~al.(2015)He, Zhang, Ren, and Sun]{7410480}
Kaiming He, Xiangyu Zhang, Shaoqing Ren, and Jian Sun.
\newblock Delving deep into rectifiers: Surpassing human-level performance on imagenet classification.
\newblock In \emph{2015 IEEE International Conference on Computer Vision (ICCV)}, pages 1026--1034, 2015.
\newblock \doi{10.1109/ICCV.2015.123}.

\bibitem[Dagan and Shamir(2018)]{daganshamir}
Yuval Dagan and Ohad Shamir.
\newblock Detecting correlations with little memory and communication.
\newblock In Sébastien Bubeck, Vianney Perchet, and Philippe Rigollet, editors, \emph{Proceedings of the 31st Conference On Learning Theory}, volume~75 of \emph{Proceedings of Machine Learning Research}, pages 1145--1198. PMLR, 06--09 Jul 2018.
\newblock URL \url{https://proceedings.mlr.press/v75/dagan18a.html}.

\bibitem[Paley and Zygmund(1932)]{paley1932some}
Raymond~EAC Paley and Antoni Zygmund.
\newblock On some series of functions,(3).
\newblock In \emph{Mathematical Proceedings of the Cambridge Philosophical Society}, volume~28, pages 190--205. Cambridge University Press, 1932.

\end{thebibliography}

\clearpage
\appendix
\onecolumn

\section{Useful Inequalities}

\begin{lemma}[Gaussian Concentration \citep{daganshamir}]
    \label{lem:gauconc}
    Let W be a random variable (RV) distributed normally with mean 0 and variance $\sigma^2$. Then, for any $w > 0$, 
    \begin{align*}
        \bb{P}\left[W \geq w\right] \leq \frac{\sigma e^{-w^2/2\sigma^2}}{w\sqrt{2\pi}}.
    \end{align*}
\end{lemma}

\begin{lemma}[\citep{paley1932some}]
    \label{lem:pz}
    Let Z be a non-negative random variable (RV) and let $\alpha \in [0, 1]$, then
    \begin{align*}
        \bb{P}\left[Z \geq \alpha\Ex{Z}\right] \geq (1 - \alpha)^2\frac{(\Ex{Z})^2}{\Ex{Z^2}}.
    \end{align*}
\end{lemma}

\section{Proofs and Discussions of Section~\ref{sec:privacyzosgd}}
\subsection{Proof of Lemma~\ref{estimator:multiplicative}}
\label{app:estimatormultiplicative}
\begin{proof}
    By the definition of SPSA estimator for $f(w)$, $\mu > 0$ and ${Z} \sim \calN(0, I_d)$, we have
    \begin{align*}
        SPSA(f, w) &= \frac{f(w + \mu {Z}) - f(w - \mu {Z})}{2\mu} {Z}.
    \end{align*}
    If $f(w) = 0$ for all $w \in \Real^d$, then $SPSA(f, w) = \mathbf{0}_d$ satisfying property 1. Moreover, for the function $f(w) = \frac{a}{2}\twonorm{w}^2$ for $a > 0$, we have:
    \begin{align*}
        SPSA(f, w) &= \frac{a}{2}\frac{\twonorm{w + \mu {Z}}^2 - \twonorm{w - \mu {Z}}^2}{2\mu} {Z} \\
        &= \frac{a}{2} \frac{\paren{\twonorm{w}^2 + \twonorm{\mu {Z}}^2 + 2\mu w^T{Z}} - \paren{\twonorm{w}^2 + \twonorm{\mu {Z}}^2 - 2\mu w^T{Z}}}{2\mu} {Z} \\
        &= a{Z}{Z}^Tw.
    \end{align*}
    Then for any $c \in \bb{R}^d$, consider its first index $\{c\}_1$, then the event $a{Z}{Z}^Tw = c$ implies that $\{a{Z}{Z}^Tw\}_1 = \{c\}_1$ i.e. $\bb{P}\left[a{Z}{Z}^Tw = c\right] \leq \bb{P}\left[\{a{Z}{Z}^Tw\}_1 = \{c\}_1\right]$. Now, we can write $\{a{Z}{Z}^Tw\}_1 = aw_1\{Z\}_1^2 + a\sum_{i=2}^d w_i\{Z\}_i\{Z\}_1$. Since $Z \sim \calN(0, I_d)$, it implies that $\{a{Z}{Z}^Tw\}_1$ is a non-constant polynomial in $\{Z\}_1, ..., \{Z\}_d$ which means that $\{a{Z}{Z}^Tw\}_1$ is continuous, implying that $\bb{P}\left[\{a{Z}{Z}^Tw\}_1 = \{c\}_1\right] = 0$ for any $\{c\}_1 \in \Real$, implying that $\bb{P}\left[a{Z}{Z}^Tw = c\right] = 0$. Hence, $SPSA(f, w)$ is a continuous RV satisfying property 2.
    
    Similarly, by definition of the FD estimator for $f(w)$, $\mu > 0$ and ${Z} \sim \calN(0, I_d)$, we get that
    \begin{align*}
        FD(f, w) &= \frac{f(w + \mu {Z}) - f(w)}{2\mu} {Z}
    \end{align*}
    If $f(w) = 0$ for all $w \in \Real^d$, then $FD(f, w) = \mathbf{0}_d$ satisfying property 1. If $f(w) = \frac{a}{2}\twonorm{w}^2$ for $a > 0$, we have that:
    \begin{align*}
        FD(f, w) &= \frac{a}{2}\frac{\twonorm{w + \mu {Z}}^2 - \twonorm{w}^2}{\mu} {Z} \\
        &= \frac{a}{2} \frac{\paren{\twonorm{w}^2 + \twonorm{\mu {Z}}^2 + 2\mu w^T{Z}} - \twonorm{w}^2}{\mu} {Z} \\
        &= \frac{a}{2}\paren{\mu\twonorm{{Z}}^2 + 2w^T{Z}} {Z}.
    \end{align*}
    Considering a similar argument as the SPSA estimator, we evaluate $\{FD(f,w)\}_1$, which is equal to $\frac{a\mu}{2}\{Z\}_1^3 + \frac{a\mu}{2} \sum_{i=2}^d \{Z\}_i^2\{Z\}_1 + aw_1\{Z\}_1^2 + a\sum_{i=2}^dw_i\{Z\}_i\{Z\}_1$. Since $\{FD(f,w)\}_1$ is a non-constant polynomial in $\{Z\}_1, ..., \{Z\}_d$, $\{FD(f,w)\}_1$ is continuous.
    Hence, $FD(f, w)$ is a continuous RV satisfying property 2.

    Similarly, by definition of the SP estimator for $f(w)$, $\mu > 0$ and ${Z} \sim \calN(0, I_d)$, we get that
    \begin{align*}
        SP(f, w) &= \frac{d}{\mu}f(w + \mu {Z})Z
    \end{align*}
    If $f(w) = 0$ for all $w \in \Real^d$, then $FD(f, w) = \mathbf{0}_d$ satisfying property 1. If $f(w) = \frac{a}{2}\twonorm{w}^2$ for $a > 0$, we have that:
    \begin{align*}
        SP(f, w) &= \frac{d}{\mu}f(w + \mu {Z})Z \\
        &= \frac{ad}{2\mu}\twonorm{w + \mu Z}^2{Z} \\
        &= \frac{a}{2}\paren{\twonorm{w}^2 + \mu^2\twonorm{{Z}}^2 + 2w^TZ} {Z}.
    \end{align*}
    Considering similar arguments as the SPSA and FD estimator, we see that $\{SP(f,w)\}_1$ is a non-constant polynomial in $\{Z\}_1, ..., \{Z\}_d$ which implies that $\{SP(f,w)\}_1$ is continuous. Hence, $SP(f, w)$ is a continuous RV satisfying property 2.
    
    For the mean extensions of the given oracles, we use the definition of the mean oracle
    \begin{align*}
        M^\calE(f, w, m) = \frac{1}{m} \sum_{i=1}^m U_i,
    \end{align*}
    where $U_i$ drawn i.i.d. from $\calE(f, w)$.
    Each $U_i$ is $\mathbf{0}_d$ when $f(w) = 0$ for all $w \in \Real^d$. Then, we get that for $f(w) = 0$ for all $w \in \Real^d$, $M^\calE(f, w, m) = \mathbf{0}_d$, satisfying property 1. In the other case, since $\calE_i(f, w)$ is continuous for all $i \in [m]$. We
    utilize the fact that if multiple continuous and independent random variables are added then their resultant addition is also a continuous random variable. Thus, the results extend for $M^{SPSA}$, $M^{FD}$, and $M^{SP}$.
\end{proof}

\subsection{Further Discussion on Zeroth Order Estimators}
Consider the estimator given by \citet{duchioptzero} for minimization of non-smooth functions.
\begin{definition}[\citet{duchioptzero}]
    \label{duchi:oracle}
    If ${Z_1}, {Z_1} \sim \calN(0, I_d)$, then the estimator is given by
    \begin{align*}
        G_{\mu_1, \mu_2}(f, w) = \frac{f(w + \mu_1{Z_1} + \mu_2{Z_2}) - f(w + \mu_1{Z_1})}{\mu_2}{Z_2}.
    \end{align*}
\end{definition}

\begin{lemma}
    The oracle defined in Definition~\ref{duchi:oracle} is a \zpn oracle.
\end{lemma}

\begin{proof}
    Using identical arguments from the proof of Theorem~\ref{estimator:multiplicative}, we see that with $f(w) = 0$ for all $w \in \Real^d$, $G_{\mu_1, \mu_2}(f, w) = \mathbf{0}_d$. Calculating for $f(w) = \frac{A}{2}\twonorm{w}^2$, we get that
    \begin{align*}
        G_{\mu_1, \mu_2}(f, w) = &= \frac{A}{2}\frac{\twonorm{w + \mu_1 {Z_1} + \mu_2 {Z_2}}^2 - \twonorm{w + \mu_1 {Z_1}}^2}{\mu_2} {Z_2} \\
        &= \frac{A}{2} \frac{\paren{\twonorm{w}^2 + \twonorm{\mu_1 {Z_1} + \mu_2 {Z_2}}^2 + 2 w^T\paren{\mu_1 {Z_1} + \mu_2 {Z_2}}} - \twonorm{w}^2}{\mu_2} {Z_2} \\
        &= \frac{A}{2\mu_2}\paren{\twonorm{\mu_1 {Z_1} + \mu_2 {Z_2}}^2 + 2w^T\paren{\mu_1 {Z_1} + \mu_2 {Z_2}}} {Z_2}.
    \end{align*}
    Thus, using arguments similar to the proof of Lemma~\ref{estimator:multiplicative} $G_{\mu_1, \mu_2}(f, w)$ would have a continuous distribution, proving our claim.
\end{proof}

The above lemma also implies that $\c{M}^{G_{\mu_1, \mu_2}}$ would also be a continuous distribution. It is evident that almost every zeroth order estimator used in the literature (yet) satisfies the properties of \zpn oracles.

\section{Proofs of Section~\ref{sec:privacyzosgdunk}}

\subsection{Proof of Theorem~\ref{multidim:varinit}}
\label{app:multidim:varinit}

\begin{definition}{(Restated Definition~\ref{def:badoracle})}
    An update oracle $\calO$ is $C_s$-\ac if there exists an index $i^* \in [d]$ such that if we consider the function class $\c{F} = \left\{\inner{g\mathbf{e}_{i^*}}{w}: - L \leq g \leq 0\right\}$, then for any $h \in \c{F}$ and $\overline{w} \in \Real^d$, if $U \sim \calO(h, \overline{w})$ then,
    \begin{enumerate}
        \item $\{\nabla h(\overline{w})\}_{i^*} = 0$ implies $U = \mathbf{0}_d$ w.p. 1 i.e. $\bb{P}\left[U = \mathbf{0}_d\right] = 1$
        \item $\{\nabla h(\overline{w})\}_{i^*} \neq 0$ implies $\bb{P}\left[\{U\}_{i^*} < 0\right] = 1$ 
        \item $\bb{E}\left[\{U\}_{i^*}\right] \leq \{\nabla h(\overline{w})\}_{i^*}$
        \item For any set of $\{w_1, w_2, ..., w_N\} \subset \Real^d$, let $U_j \sim \calO(h, w_j)$ independently for all $j \in [N]$. Then $\Ex{\paren{\sum_{j=1}^N \{U_j\}_{i^*}}^2}\leq C_s\Ex{\sum_{j=1}^N \{U_j\}_{i^*}}^2$
    \end{enumerate}
\end{definition}

\begin{theorem}{(Restated Theorem~\ref{multidim:varinit})}
     Consider running T steps of Algorithm \ref{alg:proj_zo_sgd} using a $C_s$-\acs~oracle $\calO$, as defined in Definition~\ref{def:badoracle}, with $D > \frac{\eta L}{2n}$ and $\calR_{init}$ is $\calN(0, \sigma^2I)$. Assume that the algorithm only returns the final iterate. Then, there exists an L-Lipschitz linear loss function over the set $[-D, D]^d \subset \Real^d$ such that for any $T \geq 1$ the output of Algorithm \ref{alg:proj_zo_sgd} is not $(\varepsilon, \delta)$-differentially private for any $\varepsilon, \delta$ satisfying $\delta \leq \frac{1}{16C_s}\max\left\{\frac{1}{T^{2/3}}, \frac{\eta L}{2nD}\right\}$ and
    \begin{align*}
        \varepsilon \leq \min\left\{\frac{\eta^2L^2T^{4/3}}{8n^2\sigma^2}, \frac{D^2}{2\sigma^2}\right\} + \ln\paren{\frac{\sqrt{2\pi}L\eta}{64C_sn\sigma}}.
    \end{align*}
\end{theorem}

\begin{proof}
    Consider the following loss function for a database $\calX = \{x_1, x_2, ..., x_n\}$
    \begin{align*}
        \mathcal{L}(w; \calX) = \frac{-1}{n} \inner{w}{\sum_{i = 1}^n x_i},
    \end{align*}
    where $\twonorm{x_i} \leq L$ for all $i \in [n]$ and $x_i \in \Real^d$ with $w \in [-D, D]^d$.
    Consider the neighbouring databases $\chi = \{x_1, x_2, ..., x_{n-1}, x_n\}$ and $\chi' = \{x_1, x_2, ..., x_{n-1}, x_n'\}$ differing only at the last entry. Let $x_1, ..., x_{n-1}, x_n$ to be $\mathbf{0}_d$. Given the index $i^{*}$ defined in Defintion~\ref{def:badoracle} for $C_s$-\acs~oracles, let $x_n' = L{\mathbf{e}_{i^*}}$. With this construction, we have $\calL(w; \calX) = 0$ and $\calL(w; \calX') = \frac{-L}{n}\inner{w}{{\mathbf{e}_{i^*}}}$. Let $W^{\calX}_t$ and $W^{\calX'}_t$ be the $t^{th}$ iterate of Algorithm~\ref{alg:proj_zo_sgd} is run on $\calL(\cdot, \calX)$ and $\calL(\cdot, \calX')$, respectively. 
    
    To show that the algorithm is differentially private, we need to define a measurable set $S$ so that Definition~\ref{def:dp} fails. Let
    \begin{align*}
        S = \left\{w \in \Real^d: \{w\}_{i^*} \geq \min\left\{\frac{\eta L}{2n}T^{2/3}, D\right\}  \right\}.
    \end{align*}

     We will show that for this set S, $\bb{P}\left[W^{\calX'}_T \in S\right] \geq e^{\varepsilon_0}\bb{P}\left[W^{\calX}_T \in S\right] + \delta_0$ for some $\varepsilon_0$ and $\delta_0$.
    
    \textbf{Note:} The two adjacent functions have been designed in a manner such that in one case, the iterate stays near origin with high probability. In the other case, it shifts away from its original point with a good probability, enough to ensure separation between the high probability regions of the two iterates.
    
    Based on our definition of S, we divide our analysis into two cases:
\begin{itemize}
    \item \textbf{Case 1:} $T \leq \paren{\frac{2nD}{\eta L}}^{3/2}$ or equivalently $\frac{\eta L}{2n}T^{2/3} \leq D$\\
        \textbf{Computing $\bb{P}\left[W^{\calX'}_T \in S\right]$} Since we have assumed the constraint space to be a hypercube, projection corresponds to coordinate wise clipping. Hence, it would suffice to analyze the setting for one coordinate without affecting the other coordinates and vice versa. Due to the second property in Definition~\ref{def:badoracle}, we have that $\{W^{\calX'}_t\}_{i^*}$ is monotonic, i.e. $\{W^{\calX'}_{t-1}\}_{i^*} \leq \{W^{\calX'}_{t}\}_{i^*}$ for $t \geq 2$. Notice that if $W^{\calX'}_T \notin S$ then $\{W^{\calX'}_T\}_{i^*} < D$ which implies that no projection occurred in the ${i^*}^{th}$ coordinate on the right side of the interval $[-D, D]$ in $T$ iterations. Hence, $\{W^{\calX'}_T\}_{i^*} =\max\left\{-D, \{W^{\calX'}_0\}_{i^*} - \eta \{U^{\calX'}_1\}_{i^*}\right\} - \eta \sum_{j=2}^{T}\{U^{\calX'}_j\}_{i^*}$ where $U^{\calX'}_j \sim \calO\left(\mathcal{L}(\cdot;\calX'), W^{\calX'}_{j-1}\right)$. Hence, 
        \begin{align*}
                \bb{P}\left[W^{\calX'}_T \in S\right] &= 1 -\bb{P}\left[\max\left\{-D, \{W^{\calX'}_0\}_{i^*} - \eta \{U^{\calX'}_1\}_{i^*}\right\} - \eta \sum_{j=2}^{T}\{U^{\calX'}_j\}_{i^*} < \frac{\eta L}{2n}T^{2/3}\right] \\
                &\geq 1 - \bb{P}\left[\{W^{\calX'}_0\}_{i^*} - \eta \sum_{j=1}^{T}\{U^{\calX'}_j\}_{i^*} < \frac{\eta L}{2n}T^{2/3}\right]\\
                &= \bb{P}\left[\{W^{\calX'}_0\}_{i^*} - \eta \sum_{j=1}^{T}\{U^{\calX'}_j\}_{i^*} \geq \frac{\eta L}{2n}T^{2/3}\right], \tag{a}
        \end{align*}
        where the inequality is due to the fact $\{W^{\calX'}_T\}_{i^*} \geq \{W^{\calX'}_0\}_{i^*} - \eta \sum_{j=1}^{T}\{U^{\calX'}_j\}_{i^*}$. Next, we obtain a lower bound on $\bb{P}\left[\{W^{\calX'}_0\}_{i^*} - \eta \sum_{j=1}^{T}\{U^{\calX'}_j\}_{i^*} \geq \frac{\eta L}{2n}T^{2/3} \right]$. Take $Z = -\sum_{j=1}^{T}\{U^{\calX'}_j\}_{i^*}$. We have
        \begin{align*}
            \bb{P}\left[\{W^{\calX'}_0\}_{i^*} + \eta Z \geq \frac{\eta L}{2n}T^{2/3} \right] &\geq \bb{P}\left[\eta Z \geq \frac{\eta L}{2n}T^{2/3} - \{W^{\calX'}_0\}_{i^*} \bigg\vert \{W^{\calX'}_0\}_{i^*} \geq 0\right]\bb{P}\left[\{W^{\calX'}_0\}_{i^*} \geq 0\right]\\
            &\geq \frac{1}{2}\bb{P}\left[Z \geq \frac{L}{2n}T^{2/3} \right] \\
            &\geq \frac{1}{2}\bb{P}\left[Z \geq \frac{1}{2T^{1/3}}\frac{LT}{n} \right]\\
            &\geq \frac{1}{2}\bb{P}\left[Z \geq \frac{1}{2T^{1/3}}\Ex{Z} \right]. \tag{1}
        \end{align*}
        In the first inequality, we used the fact that $\{W^{\calX'}_0\}_{i^*} \sim \calN(0, \sigma^2)$ and therefore $\bb{P}\left[\{W^{\calX'}_0\}_{i^*} \geq 0\right] = \frac{1}{2}$. In the fourth inequality, we used the third property of the $C_s$-AC oracle in Definition~\ref{def:badoracle} which (with linearity of expectation) implies that $\Ex{Z} = \Ex{-\sum_{j=1}^T \{U_j\}_{i^*}} \leq -\sum_{j=1}^T\{\nabla \mathcal{L}(W_j^{\chi'};\chi')\}_{i^*} $ and the fact that $\{\nabla \mathcal{L}(w; \chi')\}_{i^*} = -\frac{L}{n}$ for all $w \in \Real^d$, by our construction. 

        Using, the second property of $C_s$-AC oracle, we have  $Z = -\sum_{j=1}^T \{U_j\}_{i^*} \geq 0$. Thus, applying Paley-Zygmund (Lemma~\ref{lem:pz}) on random variable Z for $\alpha = \frac{1}{2T^{1/3}}$, we get
        \begin{align*}
            \bb{P}\left[Z \geq \frac{1}{2T^{1/3}}\Ex{Z} \right] & \geq \paren{1 - \frac{1}{2T^{1/3}}}^2\frac{(\Ex{Z})^2}{\Ex{Z^2}} \\
            &\geq \frac{1}{4C_sT^{2/3}}. \tag{2}
        \end{align*}
      In the second inequality, we use the fourth property of $C_s$-AC oracles as defined in Definition~\ref{def:badoracle} to get $\Ex{\paren{\sum_{j=1}^N \{U_j\}_{i^*}}^2} \leq C_s \Ex{\sum_{j=1}^N \{U_j\}_{i^*}}^2$ which implies that $\Ex{Z^2} \leq C_s (\Ex{Z})^2$ and $T \geq 1$. Combining inequalities 1 and 2, we get that,
      \begin{align*}
          \bb{P}\left[\{W^{\calX'}_0\}_{i^*} - \eta \sum_{j=1}^{T}\{U^{\calX'}_j\}_{i^*} \geq \frac{\eta L}{2n}T^{2/3} \right] \geq \frac{1}{8C_sT^{2/3}}. \tag{b}
      \end{align*}
      Combining inequalities (a) and (b), we get
      \begin{align*}
          \bb{P}\left[W^{\calX'}_T \in S\right] \geq \frac{1}{8C_sT^{2/3}}. \tag{C1}
      \end{align*}
      So far, we computed a lower bound on $\bb{P}\left[W^{\calX'}_T \in S\right]$. Next, we compute an upper bound on $\bb{P}\left[W^{\calX}_T \in S\right]$.
        \paragraph{Computing $\bb{P}\left[W^{\calX}_T \in S\right]$} Using the first property of the $C_s$-AC oracle in Definition~\ref{def:badoracle}, we have that $W^{\calX}_T = W_0 \sim \calN(0, I_d)$. Since the projection operator simply projects any value outside the interval to the edge, it does not change the inverse CDF on points which are within the interval. Then, we get that 
        \begin{align*}
            \bb{P}\left[W^{\calX}_T \in S\right] &= \bb{P}\left[\{w_0\}_{i^*} \geq \frac{\eta L}{2n}T^{2/3} \right] \\
            &\leq \frac{2n\sigma}{\sqrt{2\pi} \eta L T^{2/3}}e^{-\frac{\eta^2L^2T^{4/3}}{8n^2\sigma^2}}, \tag{C2}
        \end{align*}
        
        where the inequality is due to Gaussian Concenteration (Lemma~\ref{lem:gauconc}) . Using inequalities (C1) and (C2), we get that 
        \begin{align*}
            \frac{\bb{P}\left[W^{\calX'}_t \in S\right] - \frac{1}{16C_sT^{2/3}}}{\bb{P}\left[W^{\calX}_T \in S\right]} \geq \frac{\sqrt{2\pi}\eta L}{32n\sigma C_s}e^{\frac{\eta^2L^2T^{4/3}}{8n^2\sigma^2}}.
        \end{align*}
        Therefore,
        \begin{align*}
            \bb{P}\left[W^{\calX'}_t \in S\right] \geq e^{\frac{\eta^2L^2T^{4/3}}{8n^2\sigma^2} + \ln\paren{\frac{\sqrt{2\pi}L\eta}{32C_sn\sigma}}}\bb{P}\left[W^{\calX}_T \in S\right] + \frac{1}{16C_sT^{2/3}}.
        \end{align*}

    \item \textbf{Case 2:} $T \geq \paren{\frac{2nD}{\eta L}}^{3/2}$ or equivalently $\frac{\eta L}{2n}T^{2/3} \geq D$
    
    The approach to this case is the same Case 1. It only differs in the computation of the constants and dependence on the respective variables.
    
    \textbf{Computing $\bb{P}\left[W^{\calX'}_T \in S\right]$} Using the same argument as that in case 1, we get that
        \begin{align*}
            \bb{P}\left[W^{\calX'}_T \in S\right] &\geq \bb{P}\left[\{W^{\calX'}_0\}_{i^*} - \eta \sum_{j=1}^{T}\{U^{\calX'}_j\}_{i^*} \geq D\right]  \tag{d}
        \end{align*}
        Now, to obtain the lower bound on $\bb{P}\left[\{W^{\calX'}_0\}_{i^*} - \eta \sum_{j=1}^{T}\{U^{\calX'}_j\}_{i^*} \geq D \right]$, we use the same series of steps as those in Case 1.
        \begin{align*}
            \bb{P}\left[\{W^{\calX'}_0\}_{i^*} - \eta \sum_{j=1}^{T}\{U^{\calX'}_j\}_{i^*} \geq D \right] &\geq \frac{1}{2}\bb{P}\left[-\sum_{j=1}^{T}\{U^{\calX'}_j\}_{i^*} \geq \frac{D}{\eta} \right] \\
            &\geq \frac{1}{2}\bb{P}\left[-\sum_{j=1}^{T}\{U^{\calX'}_j\}_{i^*} \geq \frac{Dn}{\eta LT}\frac{LT}{n} \right] \\
            & \geq \frac{1}{2}\bb{P}\left[-\sum_{j=1}^{T}\{U^{\calX'}_j\}_{i^*} \geq \frac{Dn}{\eta LT}\Ex{-\sum_{j=1}^{T}\{U^{\calX'}_j\}_{i^*}} \right] \\
            & \geq \frac{1}{2}\paren{1 - \frac{Dn}{\eta LT}}^2\frac{1}{C_s} \\
            &\geq \frac{1}{2}\paren{1 - \frac{1}{2}\sqrt{\frac{\eta L}{2Dn}}}^2\frac{1}{C_s} \\
            &\geq \frac{1}{16}\frac{\eta L}{Dn}\frac{1}{C_s}. \tag{e}
        \end{align*}
        
        In the first equality, we simply used the fact that $\{W^{\calX'}_0\}_{i^*} \sim \calN(0, \sigma^2)$ and therefore $\bb{P}\left[\{W^{\calX'}_0\}_{i^*} \geq 0\right] = \frac{1}{2}$. In the third inequality, we used the third property of the $C_s$-AC oracle where $\{\nabla \mathcal{L}(\cdot; \chi')\}_{i^*} T \geq \Ex{\sum_{j=1}^T \{U_j\}_{i^*}}$ and $\{\nabla \mathcal{L}(\cdot; \chi')\}_{i^*} = -\frac{L}{n}$, by our construction. In the fourth inequality, using the second property of the $C_s$-AC oracle, we get $-\sum_{j=1}^T \{U_j\}_{i^*} \geq 0$, $T \geq 1$. Hence, we apply Paley-Zygmund on $-\sum_{j=1}^N \{U_j\}_{i^*}$ and then use the fourth property of the $C_s$-AC oracle implying that $\Ex{\paren{\sum_{j=1}^N \{U_j\}_{i^*}}^2} \leq C_s \Ex{\sum_{j=1}^N \{U_j\}_{i^*}}^2$. In the fifth inequality, we use the fact that $T \geq \paren{\frac{2nD}{\eta L}}^{3/2}$ and in the sixth inequality, we use the fact that $\frac{1}{2}\sqrt{\frac{\eta L}{2Dn}} \leq \frac{1}{2}$. Thus, combining inequalities (d) and (e), we get that
        \begin{align*}
            \bb{P}\left[W^{\calX'}_t \notin S\right] \geq \frac{1}{16}\frac{\eta L}{Dn}\frac{1}{C_s}  \tag{F1}
        \end{align*}
        
        \paragraph{Computing $\bb{P}\left[W^{\calX}_T \in S\right]$} This argument follows exactly from the first case. We get that 
        \begin{align*}
            \bb{P}\left[W^{\calX}_T \in S\right] = \bb{P}\left[\{w_0\}_{i^*} \geq D \right] \leq \frac{\sigma}{\sqrt{2\pi} D}e^{-\frac{D^2}{2\sigma^2}} \tag{F1}
        \end{align*}
        Hence, using inequalities F1 and F2, we get
        \begin{align*}\frac{\bb{P}\left[W^{\calX'}_t \in S\right] - \frac{\eta L}{32C_sDn}}{\bb{P}\left[W^{\calX}_T \in S\right]} \geq \frac{\sqrt{2\pi}\eta L}{32C_sn\sigma}e^{\frac{D^2}{2\sigma^2}} 
        \end{align*}
        
        Therefore,
        \begin{align*}
        \bb{P}\left[W^{\calX'}_t \in S\right] \geq e^{{\frac{D^2}{2\sigma^2}} + \ln\paren{\frac{\sqrt{2\pi}\eta L}{32C_sn\sigma}}}\bb{P}\left[W^{\calX}_T \in S\right] + \frac{Dn}{32C_s\eta L}
        \end{align*}
\end{itemize}

Combining the above two cases proves our claim.
\end{proof}

\section{Discussion of Oracles}
\label{app:baddiscussion}
\begin{lemma}
    SPSA, FD, and estimator defined in Definition~\ref{duchi:oracle} are 3 \acs.
\end{lemma}
\begin{proof}
    Take $i^* = 1$. Consider any $f \in \calF$. Then, the expression of $f = \inner{ge_1}{w}$, which means that $\abs{\{\nabla f\}_{i^*}} = g$. Then, for ${Z_{SPSA}}, {Z_{FD}}, {Z_{G_1}} {Z_{G_2}} \sim \calN(0, I_d)$,
    \begin{align*}
        SPSA(f, w) &= \frac{\inner{ge_1}{w + \mu {Z_{SPSA}}} - \inner{ge_1}{w - \mu {Z_{SPSA}}}}{2\mu}{Z_{SPSA}} \\
        &= g\{{Z_{SPSA}}\}_1{Z_{SPSA}} \\
        FD(f, w) &= \frac{\inner{ge_1}{w + \mu {Z_{FD}}} - \inner{ge_1}{w}}{\mu}{Z_{FD}} \\
        &= g\{{Z_{FD}}\}_1{Z_{FD}}\\
        G_{\mu_1, \mu_2}(f, w) &= \frac{\inner{ge_1}{w + \mu_1 {Z_{G_1}} + \mu_2 {Z_{G_2}}} - \inner{ge_1}{w + \mu_1 {Z_{G_1}}}}{\mu_2}{Z_{G_2}}\\
        &= g\{{Z_{G_2}}\}_1{Z_{G_2}}
    \end{align*}
    Since ${Z_{SPSA}}, {Z_{FD}}, {Z_{G_2}}$ are i.i.d. random variables, it implies that the given estimators follow the same distribution. Take $\c{E}$ to be any one of the oracles, and for ${Z} \sim \calN(0, I_d)$, we have that $\c{E}(f, w) = g\{{Z}\}_1{Z}$. Thus, we have that $\{\c{E}(f, w)\}_1 = g\{{Z}\}_1^2$ and $\{{Z}\}_1^2 \sim \chisq(1)$. $\{U^{(f)}\}_1 = g{V}$ where ${V} \sim \chisq(1)$. Now, we verify the properties
    \begin{enumerate}
        \item Observe that when $\abs{\{\nabla f\}_{i^*}} = g = 0$ which implies that $\c{E}(f,w) = \mathbf{0}_d$, satisfying property 1.
        \item For the second property, ${V} \geq 0$ which implies that $gV \leq 0$ for all .
        \item $\Ex{V} = 1$ which implies that $\Ex{\{U^{(f)}\}_{i^*}} = g\Ex{V} = g = \abs{\{\nabla f\}_{i^*}}$.
        \item $R = \sum_{j=1}^N V_j \sim \chisq(N)$. Hence, we know that $\Ex{R} = N$ and $Var[R] = 2N$, which implies that $\Ex{R^2} = N^2 + 2N$. We also have that $\sum_{j=1}^N \{U^{(f)}_j\}_{i^*} = g\sum_{j=1}^N V_j$, which implies that $\Ex{\paren{\sum_{j=1}^N \{U^{(f)}_j\}_{i^*}}^2} = g^2(N^2 + 2N)$ and $\Ex{\paren{\sum_{j=1}^N \{U^{(f)}_j\}_{i^*}}}^2 = g^2N^2$. Using $N \geq 1$ gives the value of $C_s = 3$ proving our given claim.
    \end{enumerate}
\end{proof}

\begin{lemma}{(Restated)}
    If an oracle $\c{E}$ is $C_s$-\acs~ then $M_m^{\c{E}}$ is $C_s$-\acs.
\end{lemma}
\begin{proof}
    If $\calE$ is $C_s$-AC, then there exists an $i^* \in [N]$ which satisfies the properties mentioned in Definition~\ref{def:badoracle}. Using definition of mean extensions of estimators, we have that for $f \in \calF$ (as defined in Definition~\ref{def:badoracle})
    \begin{align*}
        M^\calE_m(f, w) = \frac{1}{m} \sum_{j=1}^m U_j^{(f,w)},
    \end{align*}
    where $U_j^{(f,w)}$ is drawn i.i.d. from $\calE(f, w)$. 

    Using property 1 of $\calE$, $\abs{\{\nabla f\}_{i^*}} = 0$ implies $U_j^{(f,w)} = \mathbf{0}_d$ for all $j \in [m]$. Thus, we get $\{M^\calE_m(f, w)\}_{i^*} = \frac{1}{m} \sum_{j=1}^m \{U_j^{(f,w)}\}_{i^*} < 0$, satisfying property 2.

    Similarly, for property 2, $\abs{\{\nabla f\}_{i^*}} \neq 0$ implies that $\{U_j^{(f,w)}\}_{i^*} < 0$ for all $j \in [m]$. Thus, we get $\{M^\calE_m(f, w)\}_{i^*} = \frac{1}{m} \sum_{j=1}^m \{U_j^{(f,w)}\}_{i^*} < 0$, satisfying property 2.
    
    For the third property, consider $\Ex{\{M^\calE_m(f, w)\}_{i^*}}$. Using linearity of expectation, we get
    \begin{align*}
        \Ex{\frac{1}{m}\sum_{j=1}^m \{U_j^{(f,w)}\}_{i^*}} &= \frac{1}{m}\sum_{j=1}^m \Ex{\{U_j^{(f,w)}\}_{i^*}} \\
        &\leq \{\nabla h\}_{i^*}.
    \end{align*}
    Hence, $M^\calE_m$ satisfies property 3. 
    
    Consider the set $\{w_{c}\}_{c=1}^{N}$. By the definition of $M^\calE_m$, we have 
    \begin{align*}
        \Ex{\paren{\frac{1}{m}\sum_{k=1}^N\sum_{j=1}^m \{U_j^{(f,w_k)}\}_{i^*}}^2} = \Ex{\paren{\frac{1}{m}\sum_{j=1}^m\sum_{k=1}^N \{U_j^{(f,w_k)}\}_{i^*}}^2}
    \end{align*}
    Let $S_j = \sum_{k=1}^N \{U_j^{(f, w_k)}\}_i$. Since $U_j^{(f, w_k)}$ are sampled i.i.d. from $\calE(f, w_k)$ for all $j \in [m]$, it implies that $S_j = \sum_{k=1}^N \{U_j^{(f, w_k)}\}_{i^*}$ is identically distributed for all $j \in [m]$. Thus, we can take $\Ex{S_u} = K_f$, and $\Ex{S_u^2} = K_s$ for all $u \in [m]$. Applying Young's inequality, we get that
    \begin{align*}
        \Ex{\paren{\frac{1}{m}\sum_{u=1}^mS_u}^2} &\leq \Ex{\frac{1}{m}\sum_{u=1}^m S_u^2} \\
        &= K_s
    \end{align*}
    Using property 4 of $\calE$, we have $K_s \leq C_s K_f^2$. Thus, substituting the original terms, we have
    \begin{align*}
        \Ex{\paren{\frac{1}{m}\sum_{k=1}^N\sum_{j=1}^m \{U_j^{(f,w_k)}\}_{i^*}}^2} \leq C_s \Ex{\sum_{k=1}^N \{U_u^{(f, w_k)}\}_{i^*}}^2  \tag{g}
    \end{align*}
    Thus, using the fact that $U_j^{(f, w_k)}$ is sampled i.i.d. for a $k \in [N]$, we see that $\Ex{\sum_{k=1}^N \{U_u^{(f, w_k)}\}_{i^*}} = \Ex{\frac{1}{m}\sum_{j=1}^m\sum_{k=1}^N \{U_j^{(f, w_k)}\}_{i^*}}$. Substituting this in the RHS of (g), we get that $M^\calE_m$ satisfies the final property.
\end{proof}

\end{document}